\documentclass[oneside]{amsart}

\usepackage[a4paper, total={17cm, 26.5cm}]{geometry}
\usepackage{amsmath}
\usepackage{amssymb}
\usepackage{xcolor, colortbl}
\usepackage{hyperref}
\usepackage{float}
\usepackage{longtable}
\usepackage{graphicx}
\usepackage{enumerate}

\def\Coin{\operatorname{Coin}}
\def\N{\mathbb{N}}
\def\Z{\mathbb{Z}}
\def\Q{\mathbb{Q}}
\def\R{\mathbb{R}}
\def\C{\mathbb{C}}
\def\T{\mathbb{T}}
\def\Id{\operatorname{Id}}

\def\Diff{\operatorname{Diff}}
\def\Fix{\operatorname{Fix}}

\def\LCM{\operatorname{LCM}}
\def\Ker{\operatorname{Ker}}
\def\Im{\operatorname{Im}}
\def\Coker{\operatorname{Coker}}
\def\ord{\operatorname{ord}}
\def\tr{\operatorname{tr}}
\def\Int{\operatorname{Int}}

\def\divides{{\mathchoice{\mathrel{\bigm|}}{\mathrel{\bigm|}}{\mathrel{|}}{\mathrel{|}}}}
\def\Divides{\divides\!\divides}
\newcommand{\wh}[1]{\widehat{#1}}

\newtheorem{thm}{Theorem}
\newtheorem{lem}[thm]{Lemma}
\newtheorem{ex}[thm]{Example}
\newtheorem{rmk}[thm]{Remark}

\newtheorem{col}[thm]{Corollary}

\title[Synchronization zeta function]{Synchronization points: growth, asymptotics, congruences, and the synchronization zeta function}

\author{Alexander Fel'shtyn and Mateusz Slomiany}

\address{Alexander Fel'shtyn, Instytut Matematyki, Uniwersytet Szczecinski,
ul. Wielkopolska 15, 70-451 Szczecin, Poland}

\address{Mateusz Slomiany, Instytut Matematyki, Szkola Doktorska, Uniwersytet Szczecinski,
ul. Wielkopolska 15, 70-451 Szczecin, Poland}

\begin{document}

\begin{abstract}
In this paper, we introduce the synchronization zeta function associated with a pair of self-maps
of a topological space and investigate its properties.
We also define   the growth rate of synchronization points and derive an explicit formula
in the setting of endomorphisms of compact, connected Abelian groups.
In addition, we establish Gauss congruences and describe the asymptotic behavior for the sequence
of numbers of synchronization points, under the assumption that the synchronization zeta function is rational.
Further, we  discuss connections with topological entropy.
\end{abstract}
\maketitle

\setcounter{tocdepth}{1}
\tableofcontents

\vspace{1cm}
\bigskip
\section{Introduction}

Let $X$ be a topological space and let $\alpha, \beta: X \to X$ be two maps of $X$ to itself.
We denote the set of \emph{synchronization points of $n$-th iteration of $\alpha$ and $\beta$} by
\begin{equation*} 
  S(\alpha^n, \beta^n) := \left\{ x \in X \ \vert \ \alpha^n(x) = \beta^n(x) \right\}.
\end{equation*}
If one of the maps is the identity map, say $\beta = \Id$, then $S(\alpha^n, \Id)$ is the set of
fixed points of $\alpha^n$.
In topology the synchronization points are also known as \emph{points of coincidence}.

We say that a pair of maps $\alpha, \beta$ is \emph{synchronously tame} if
the numbers of synchronization points $\# S(\alpha^n, \beta^n) $  are finite for all $n \in \N$.

In the present paper we define the \emph{synchronization zeta function} of a synchronously tame pair $\alpha, \beta$ as
\begin{equation*} 
  S_{\alpha, \beta}(z) := \exp \left( \sum_{k=1}^\infty \frac{\# S(\alpha^k, \beta^k)}{k} z^k \right),
\end{equation*}
where $z$ denotes a complex variable, and undertake its investigation.
When $\beta$ or $\alpha$ is the identity map, the synchronization zeta function becomes
the \emph{Artin-Mazur zeta function}:
\begin{equation*}
  \zeta_\alpha(z) := \exp \left( \sum_{n=1}^\infty \frac{\# \Fix(\alpha^n)}{n} z^n \right).
\end{equation*}
The synchronization zeta function can be
regarded as an analogue of the Hasse--Weil zeta function of an
algebraic variety over a finite field or the Artin--Mazur zeta
function of a continuous self-map of a topological space.
In the theory of dynamical systems, the synchronization zeta
function counts the synchronisation points of two maps, i.e. the
points whose orbits intersect under simultaneous iteration of two
endomorphisms; see \cite{Miles13}, for instance.

We define the \emph{upper growth rate}, \emph{lower growth rate}
and \emph{growth rate} (if it exists) of synchronization points of $\alpha$ and $\beta$, to be
\begin{align*} 
\begin{split}
  S^\infty_+(\alpha, \beta) &:= \limsup_{n \to \infty} \left(\# S(\alpha^n, \beta^n) \right)^{1/n}, \\
  S^\infty_-(\alpha, \beta) &:= \liminf_{n \to \infty} \left( \# S(\alpha^n, \beta^n) \right)^{1/n}, \\
  S^\infty(\alpha, \beta) &:= \lim_{n \to \infty} \left( \# S(\alpha^n, \beta^n) \right)^{1/n},
\end{split}
\end{align*}
respectively.
Due to Cauchy--Hadamard theorem, the radius of convergence of the synchronization zeta function
is the inverse $1/S^\infty_+(\alpha,\beta)$ of the upper growth rate.

In this work  we present, in analogy to works of
Bell, Miles, Ward~\cite{BMW} and Byszewski, Cornelissen~\cite[\S
5]{ByCo18} and ~\cite{FelsKlop} results in support of a P\'olya--Carlson dichotomy between
rationality and a natural boundary for the analytic behavior of the
synchronization  zeta function.

In case of compact connected Abelian groups we give an explicit formula
for the growth rate $S^\infty(\alpha,\beta)$.
In case of some homeomorphisms of compact metric spaces we show a relationship between
the growth rate $S^\infty(\alpha,\beta)$ and topological entropy.
We also show that a positive answer to the question  rationality of the synchronization zeta function implies  the Gauss congruences
for the sequence $\{\# S(\alpha^n,\beta^n) \}$.
We provide an observation on asymptotic behavior of the sequence $\{\# S(\alpha^n,\beta^n) \}$.

This work is organized as follows.

In section \ref{sec:preliminaries}, we remind definitions, facts and notation
from other theories which are of importance in course of the paper.
They regard Reidemeister zeta function, P\'olya-Carlson dichotomy, unitary dual of a group,
Markov partitions, and $p$-adic numbers.

In section \ref{sec:dichotomy}, we prove  P\'olya-Carlson dichotomy  for the
synchronization  zeta function of  endomorphisms of compact connected Abelian groups
as well as rationality of this zeta function  for dual maps to commuting automorphisms
of torsion-free virtually polycyclic groups, for commuting Axiom A diffeomorphisms
of $C^\infty$ compact manifolds without boundary, and for commuting pseudo-Anosov homeomorphisms
of closed compact surfaces.
We also investigate the synchronization zeta function of maps of any finite set.

In section \ref{sec:congruences}, we prove the Gauss congruences for the sequence
$\{\# S(\varphi^n,\psi^n) \}$ of the  numbers of synchronization points of the synchronously
tame pair $(\varphi,\psi)$ of maps, under the assumption that the synchronization zeta function is rational.

In section \ref{sec:entropy}, we first discuss expansiveness and the so called \emph{specification} property
for homeomorphisms, then we provide a relation between the growth rate of synchronization points
of a pair $\alpha, \beta$ of commuting homeomorphisms of compact metric spaces and topological entropy
of $\beta^{-1}\alpha$, when the latter satisfies specification.

In section \ref{sec:asymptotic}, we study asymptotic behavior of the sequence
$\{\# S(f^n, g^n) \}$ of the numbers of synchronization points of the synchronously
tame pair $( f, g )$ of  maps of a compact metric space when the synchronization
zeta function is rational.

\vspace{1cm}
\bigskip
\section{Preliminaries} \label{sec:preliminaries}

\subsection{The Reidemeister  and the Nielsen  coincidence numbers}

Let $f, g: X \to X$ be two maps of  compact connected manifold $X$.
Let
$\Coin(f,g) := \{ x \in X \mid f(x) = g(x) \}$ be the set of the coincidence  points of $f$ and $g$ in topological terminology   i. e. the
set $ S(f,g) $ of synchronization points of  of $f$ and $g$ in the  terminology above.
Two points of coincidence $x_1,x_2 \in \Coin(f,g)$ are said to be \emph{Nielsen equivalent} if there is a path
$\alpha:[0,1] \to X$ from $x_1$ to $x_2$ such that $f \circ \alpha$ is homotopic to $g \circ \alpha$.
If $[C]$ denotes the fixed-endpoint homotopy class of $C$
then we can denote it by $[f \circ \alpha] = [g \circ \alpha]$.
Being Nielsen equivalent is obviously an equivalence relation.
We denote the set of the coincidence classes by $\overline{\Coin}(f,g)$.
Let $\{f_t\}$ be a homotopy of $f$ and let $\{g_t\}$ be a homotopy of $g$, $f_0 = f, g_0 = g$.
Two points of coincidence $x_1 \in \Coin(f,g)$ and $x_2 \in \Coin(f_1, g_1)$ are $\{f_t\},\{g_t\}$-related
if there exists a~path $C:[0,1] \to X$ with $C(0) = x_1, C(1) = x_2$ such that
$[\{f_t \circ C(t)\}] = [\{g_t \circ C(t)\}]$.
Two equivalence classes $\alpha \in \overline{\Coin}(f,g)$ and $\beta \in \overline{\Coin}(f_1,g_1)$
are $\{f_t\},\{g_t\}$-related if every point of coincidence in $\alpha$ is $\{f_t\},\{g_t\}$-related
to every point of coincidence in $\beta$.

A coincidence class $\alpha \in \overline{\Coin}(f,g)$ is \emph{essential}
if for any homotopies $\{f_t\}$ of $f$ and $\{g_t\}$ of $g$, there is a coincidence class
in $\overline{\Coin}(f_1,g_1)$ to which $\alpha$ is $\{f_t\},\{g_t\}$-related.
The finite number of essential coincidence classes in $\overline{\Coin}(f,g)$ is called
the \emph{Nielsen coincidence number} of $f,g$ and denoted by $N(f,g)$,
see  \cite{ Wong}. The Nielsen coincidence number $N(f,g)$ is a homotopy invariant which is
a lower bound for the number of coincidence points: $N(f,g) \leq \#\Coin(f,g)=\#S(f,g)$.
Let $f, g: X \to X$ be two maps of  compact connected  orientable  $n$-manifold $X$, $n\geq 3$, then
the Nielsen coincidence number $N(f,g)$ is a sharp lower bound for the minimal number of coincidences in the homotopy class of $f$ and of $g$, see \cite{BrownSchirmer}, \cite{Je}, \cite{GoWon}.

Let $G$ be a  group and $\varphi: G\rightarrow G$ an
endomorphism. Two elements $x,y\in G$ are said to be
$\varphi$-{\em conjugate} or {\em twisted conjugate,}
if and only if there exists $g \in G$ such that $ y=g  x   \varphi(g^{-1}).$
We will write $\{x\}_\varphi$ for the $\varphi$-{\em conjugacy} or
{\em twisted conjugacy} class
of the element $x\in G$.
The number of $\varphi$-conjugacy classes is called the {\em Reidemeister number}
of an  endomorphism $\varphi$ and is  denoted by $R(\varphi)$.
If $\varphi$ is the identity map then the $\varphi$-conjugacy classes are the usual
conjugacy classes in the group $G$.

We have the Reidemeister bijection between  fixed point classes, lifting classes and corresponding twisted conjugacy
classes. Therefore the
Reidemeister number of continuous map  $f$ coincides with Reidemeister number
(the number of twisted conjugacy classes) of induced endomorphism $f_*$ of the fundamental group of $X$: $R(f)=R(f_*)$.
In this paper we take mostly a group-theoretic point of view,
inspired by, but otherwise  independent of the topological
origins of the subject.

Let $G$ be a  group and $\varphi, \psi : G\rightarrow G$
two endomorphisms. Two elements $\alpha,\beta\in G$ are said to be $(\varphi, \psi)$\textit{-conjugate} if and only if there exists $g \in G$ with
$
\beta=\psi(g)  \alpha   \varphi(g^{-1}).
$
The number of $(\varphi,\psi$)-conjugacy classes is called the Reidemeister coincidence
number of   endomorphisms $\varphi $ and $\psi $, denoted by $ R(\varphi,\psi)$. If $\psi$ is the identity map then the $(\varphi,id)$-conjugacy classes are the $\varphi$-conjugacy classes in the group $G$ and $R(\varphi,id) = R(\varphi)$.
The Reidemeister coincidence number $R(\varphi,\psi)$ has useful applications in Nielsen--Reidemeister  coincidence theory.
We call an  endomorphism $\varphi$  \emph{tame} if the
Reidemeister numbers $R(\varphi^n)$ are finite for all $n \in
\mathbb{N}$.
We call the pair $(\varphi,\psi)$ of endomorphisms \emph{tame} if the coincidence
Reidemeister numbers $R(\varphi^n,\psi^n)$ are finite for all $n \in
\mathbb{N}$.

Maps $f,g$ induce homomorphisms $f_\#, g_\#$  of  fundamental group of compact manifold $X$.
We define the \emph{topological} Reidemeister coincidence number  $R(f,g)$ of maps  $f,g$
to be the Reidemeister coincidence number of the~induced homomorphisms  $R(f,g) := R(f_\#,g_\#)$.

For selfmaps of compact solvmanifolds, we have the  principal relation between the Reidemeister  and the Nielsen  coincidence numbers i.e.
  $ N(f,g)=R(f,g)=R(f_\#,g_\#)$,  see \cite{GoWon}.

 We are going to say that the pair of maps $f,g$ is \emph{tame} if
 $R(f^n, g^n) < \infty$ for every $n \in \N$.

In the theory of dynamical systems, the Reidemeister coincidence
numbers count the synchronisation points of two
maps, i.e. the points whose orbits intersect under simultaneous
iterations of two maps; see \cite{Miles13}, for instance.

\subsection{Reidemeister zeta function and P\'olya-Carslon dichotomy}

The \emph{Reidemeister coincidence zeta function} of the tame pair $\varphi, \psi$ is defined as
\begin{equation*}
  R_{\varphi,\psi}(z) := \exp \left( \sum_{n=1}^\infty \frac{R(\varphi^n, \psi^n)}{n} z^n \right).
\end{equation*}
The \emph{growth rate} of Reidemeister coincidence numbers, if it exists, is defined as
\begin{equation*}
  R^\infty(\varphi, \psi) := \lim_{n \to \infty} R(\varphi^n, \psi^n)^{1/n}.
\end{equation*}

The classical P\'olya--Carlson theorem, as discussed in~\cite[\S 6.5]{Segal},
provides the following connection between the arithmetic properties of
the coefficients of a complex power series and its analytic behavior.
Suppose that an analytic function $F$ is defined somehow in a region $D$ of a complex plane.
If there is no point of the boundary $\partial D$ of $D$ over which $F$ can be analytically continued,
then $\partial D$ is called a \emph{natural boundary} for $F$.

\begin{thm}[P\'olya-Carlson, see \protect{\cite{Segal}}] \label{thm:polya-carlson}
  A power series with integer coefficients and radius of convergence $1$ is either rational
  or has the unit circle as a natural boundary.
\end{thm}
It is convenient to introduce a notation
\begin{equation*}
  Z_{\varphi, \psi}(z) := \sum_{n=1}^{\infty} R(\varphi^n, \psi^n)\cdot z^n.
\end{equation*}
for the generating series that enumerates directly the numbers of coincidence Reidemeister classes.
The following lemma shows in particular that, if $Z_{\varphi,\psi}(z)$ has a natural
boundary at its radius of convergence, then so
does~$R_{\varphi,\psi}(z)$; compare~\cite{BMW}.

\begin{lem} \label{thm:logarithmic-derivative}
  Let $\varphi, \psi: G \to G$ be endomorphisms of a group.
  If $R_{\varphi, \psi}(z)$ is rational then $Z_{\varphi, \psi}(z)$ is rational.
  If $R_{\varphi, \psi}(z)$ has an analytic continuation beyond
  its circle of convergence, then so does $Z_{\varphi, \psi}(z)$ too.
  In particular, the existence of a natural boundary
  at the circle of convergence for~$Z_{\varphi, \psi}(z)$ implies the
  existence of a natural boundary for $R_{\varphi, \psi}(z)$.
\end{lem}
\begin{proof}
  This follows from the fact that $Z_{\varphi, \psi}(z) = z \cdot R_{\varphi, \psi}(z)^{'}/R_{\varphi, \psi}(z)$.
\end{proof}

For the proofs of the main theorems in this paper we rely on the
following key result of Bell, Miles and Ward.
\begin{lem}[\protect{\cite[Lemma 17]{BMW}}] \label{thm:BMW-lemma17}
  Let $S$ be a finite list of places of algebraic number fields and, for each $v \in S$,
  let $\xi_v$ be a non-unit root in the appropriate number field such that $|\xi_v|_v = 1$.
  Then the function
  \begin{align*}
    F(z) := \sum_{n \geq 1} f(n) z^n,
  \end{align*}
  where $f(n) := \prod_{v \in S} |\xi_v^n-1|_v$ for $n \geq 1$,
  has the unit circle as a natural boundary.
\end{lem}

\subsection{$p$-adics}

In this section we remind basic definitions and prove facts about $p$-adic numbers needed later.
Let $p$ be any prime number. For $a \in \Z \setminus \{0\}$ we define the $p$-adic ordinal $\ord_p a$
to be the highest power of $p$ which divides $a$.
For a rational $q = a/b$ we define $\ord_p q := \ord_p a - \ord_p b$.
The $p$-adic norm of $q$ is defined
\[
  \left| q \right|_p := \begin{cases} \frac{1}{p^{\ord_p q}}, \quad &q \neq 0,\\ 0, \quad &q=0 \end{cases}.
\]
The norm is \emph{non-Archimedean} which means that for all $x, y \in \Q$
it satisfies $\left| x+y \right|_p \leq \max\left( |x|_p, |y|_p \right)$.

We will make use of the following property:
\begin{equation}\label{eq:padic-triangle}
  \big( |x|_p < |y|_p \big) \quad \Longrightarrow \quad |x-y|_p = |y|_p.
\end{equation}
The field $\Q$ with the norm $|\cdot|_p$ is not complete but if we consider the set $\Q_p$
of equivalence classes of Cauchy sequences with regard to the relation
\[
  \{a_i\} \sim \{b_i\} \quad \Longleftrightarrow \quad |a_i - b_i|_p \to 0 \ (i \to \infty)
\]
(just as we actually do when we go from $\Q$ to $\R$ with the usual norm)
with the $p$-adic norm of a sequence defined as $\left|\{a_i\}\right|_p := \lim |a_i|_p$
then $\Q_p$ turns out to be a complete field.
Moreover, from (\ref{eq:padic-triangle}) and definition of Cauchy sequences we get
that in fact the sequence $\{|a_i|_p\}$ of $p$-adic norms of elements of a Cauchy sequence $\{a_i\}$
is constant for sufficiently large $i$.
The following remark will be important later.
\begin{rmk}\label{thm:possible-values-of-padic-norm}
  Possible values for $|\cdot|_p$ on $\Q_p$ are in the set $\{p^n\}_{n \in \Z} \cup \{0\}$.
  In particular the greatest possible value $<1$ is $1/p$.
\end{rmk}

We are going to work with numbers algebraic over $\Q_p$.
Let $\Q_p(\xi)$ be a finite extension of $\Q_p$ generated by $\xi$ which is a root of a monic polynomial
\[
  x^n + a_1 x^{n-1} + \dots + a_{n-1} x + a_n, \quad a_i \in \Q_p.
\]
We define the \emph{norm} of the extension as $\N_{\Q_p(\xi)/\Q_p}(\xi) := (-1)^n a_n$.
The extension of $p$-adic norm from $\Q_p$ to $\Q_p(\xi)$ is defined as
\[
  |\xi|_p := \left| \N_{\Q_p(\xi)/\Q_p}(\xi) \right|_p^{1/n}.
\]
There exist the algebraic closure $\overline{\Q_p}$ and its completion $\Omega$ (which is algebraically
closed as~well) with further extensions of the norm but we will not need them in this work.
The subset $\Z_p := \{ x \in \Q_p \mid |x|_p \leq 1 \}$ is a ring
and its elements are called \emph{$p$-adic integers}

We are going to make use of the very specific inequality
which comes from \cite[Theorem~6.1]{CEW}.

\begin{lem}[\protect{\cite[Lemma 2.4]{FelsSlo24}}] \label{thm:FelSlo-lemma2.4}
  Let $\xi$ be an algebraic number over $\Q_p$ such that $|\xi|_p = 1$ but $\xi$ is not a root of unity.
  Then there is a constant $C$ such that for every $n \in \N$ the inequality
  \begin{equation} \label{eq:1/n-leq-x-1}
    \frac{1}{n} \leq C \cdot \left| \xi^n - 1 \right|_p
  \end{equation}
  holds. The constant does not depend on $n$.
\end{lem}

\subsection{Unitary dual}

By $\wh G$ we denote the \emph{unitary dual} of a group $G$, i.e.
the space of equivalence classes of
unitary irreducible representations of $G$, equipped with the
\emph{compact-open} topology.
If $\varphi: G\to G$ is an automorphism,
it induces a dual map $\wh\varphi:\wh G\to\wh G$, $\wh\varphi (\rho) := \rho \circ \varphi$.
This dual map $\wh\varphi$ define a dynamical system  on the unitary dual space $\widehat{G}$.

The unitary dual of any finitely generated abelian group $G$ is an abelian group $\wh G$ which is,
by Pontryagin duality, the direct sum
of a~Torus whose dimension is the rank of $G$, and of a~finite abelian group.
In that case the dual map will be called the \emph{Pontryagin dual map}.

We will need the following result.

\begin{lem} \cite[Proposition 1]{FelHil} \label{thm:FelHil-dual}
  Let $\varphi:G\to G$ be an endomorphism of an Abelian group $G$
  and let $\wh \varphi: \wh G \to \wh G$ be the map induced by $\varphi$
  on the Pontryagin dual $\wh G$ of $G$.
  Then
  \begin{equation*}
    \Ker \wh \varphi \ \cong \ \wh \Coker \varphi.
  \end{equation*}
\end{lem}

\subsection{Markov partitions}

Choose a natural number $N \geq 2$ and consider the space
\begin{equation*}
  \Omega_N \ := \ \left\{ \omega = (\ldots,\omega_{-1},\omega_0,\omega_1,\ldots) \; | \; \omega_i \in \{0,1,\ldots,N-1\}, i \in \Z \right\}.
\end{equation*}
The \emph{left shift} in $\Omega_N$ (sometimes called a \emph{topological Bernoulli shift}) is the map
\begin{equation*}
  \sigma_N: \Omega_N \to \Omega_n, \quad \sigma_N(\omega) = \omega' = (\ldots,\omega_0',\omega_1',\ldots),
\end{equation*}
where $\omega_n' = \omega_{n+1}$.

Let $A = (a_{ij})$ be an $N \times N$ matrix whose entries $a_{ij} \in \{0,1\}$.
Let
\begin{equation*}
  \Omega_A \ := \ \left\{ \omega \in \Omega_N \; | \; a_{\omega_n\omega_{n+1}} = 1, n \in \Z \right\}.
\end{equation*}
One can think that the matrix $A$ determines which symbols $0,1,\ldots,N-1$ can be neighbors
in the sequence $\omega$.
The set $\Omega_A$ is shift invariant.
The restriction $\sigma_A := \left. \sigma_N \right|_{\Omega_A}$ is a \emph{subshift of finite type}
(sometimes called a \emph{topological Markov chain}).

\begin{thm}  \cite{BowLan}\label{thm:trace-formula-for-subshifts}
   Number of periodic points of period $n \in \N$ of a subshift of finite type $\sigma_A$
  \begin{equation*}
    \#\Fix(\sigma_A^n) = \tr A^n.
  \end{equation*}
\end{thm}

\vspace{1cm}
\bigskip

\section{Synchronization zeta function and growth of synchronization points} \label{sec:dichotomy}

\subsection{P\'olya-Carslon dichotomy}
A group $G$ has
\emph{finite (Pr\"ufer) rank} if there exists an integer $r$ such that
every finitely generated subgroup of $G$ can be generated by $r$
elements; the least such $r$ is called the \emph{(Pr\"ufer) rank}
of~$G$.

Our main result for the synchronization zeta function $S_{\alpha, \beta}(z)$
in Theorem~\ref{thm:synchronization-dichotomy} below
is inspired by and proved in analogy with the following

\begin{thm}[\protect{\cite[Proposition 3.4]{FelsKlop}}] \label{thm:FelsKlop-3.4}
  Let $G$ be a torsion-free abelian group of finite rank~$d\geq 1$,
  and let $(\varphi,\psi)$ be a tame pair of endomorphisms for~$G$.
  Let $\varphi_\mathbb{Q}, \psi_\mathbb{Q} \colon G_\mathbb{Q} \to
  G_\mathbb{Q}$ denote the extensions of $\varphi, \psi$ to the
  divisible hull $G_\mathbb{Q} = \mathbb{Q} \otimes_\mathbb{Z} G \cong \mathbb{Q}^d$
  of~$G$, and suppose that the endomorphisms
  $\varphi_\mathbb{Q}, \psi_\mathbb{Q}$ are simultaneously
  triangularisable.  Let $\xi_1, \ldots, \xi_d$ and
  $\eta_1, \ldots, \eta_d$ denote the eigenvalues of
  $\varphi_\mathbb{Q}$ and $\psi_\mathbb{Q}$ in a fixed algebraic
  closure of the field~$\mathbb{Q}$, including multiplicities,
  ordered so that, for $n \in \mathbb{N}$, the eigenvalues of
  $\varphi_\mathbb{Q}^{\, n} - \psi_\mathbb{Q}^{\, n}$ are
  $\xi_1^{\, n} - \eta_1^{\, n}, \ldots, \xi_d^{\, n} - \eta_d^{\,
    n}$.  Set
  $L = \mathbb{Q}(\xi_1, \ldots, \xi_d, \eta_1, \ldots, \eta_d)$.

  For each $p \in \mathbb{P}$, we fix an embedding
  $L \hookrightarrow \overline{\mathbb{Q}}_p$, where
  $\overline{\mathbb{Q}}_p$ denotes the algebraic closure of
  $\mathbb{Q}_p$, and we write $\lvert \cdot \rvert_p$ for the unique
  prolongation of the $p$-adic absolute value
  to~$\overline{\mathbb{Q}}_p$; this is a way of choosing a particular
  extension of the $p$-adic absolute value to~$L$.  Likewise, we
  choose an embedding $L \hookrightarrow \mathbb{C}$ and denote by
  $\lvert \cdot \rvert_\infty$ the usual absolute value
  on~$\mathbb{C}$.
    \smallskip
    Then there exist subsets $I(p) \subseteq \{1,\ldots,d\}$, for
  $p \in \mathbb{P}$, such that the following hold.
  \begin{enumerate}
  \item[$(1)$] For each $p \in \mathbb{P}$, the polynomials
    $\prod_{i \in I(p)} (X - \xi_i)$ and
    $\prod_{i \in I(p)} (X - \eta_i)$ have coefficients in
    $\mathbb{Z}_p$; in particular,
    $\lvert \xi_i \rvert_p , \lvert \eta_i \rvert_p \le 1$ for
    $i \in I(p)$.
  \item[$(2)$] For each $n \in \mathbb{N}$,
  \begin{equation} \label{eq:Reidemeister-coincidence-product-eigenvalues}
    R(\varphi^n,\psi^n) = \prod_{p \in \mathbb{P}} \prod_{i \in I(p)}
    \lvert \xi_i^{\, n} - \eta_i^{\, n} \rvert_p^{\, -1} =
    \prod_{i=1}^d \lvert \xi_i^{\, n} - \eta_i^{\, n} \rvert_\infty \cdot
    \prod_{p \in \mathbb{P}} \prod_{i \not \in I(p)} \lvert \xi_i^{\, n} -
    \eta_i^{\, n} \rvert_p;
  \end{equation}
    as this number is a positive integer,
    $\lvert \xi_i^{\, n} - \eta_i^{\, n} \rvert_p = 1$ for
    $1 \le i \le d$ for almost all~$p \in \mathbb{P}$.
  \end{enumerate}
\end{thm}

Now we state and prove our main result.

\begin{thm} \label{thm:synchronization-dichotomy}
  Let $\alpha, \beta: X \to X$ be a synchronously tame pair of endomorphisms
  of a compact connected Abelian group of dimension $d \geq 1$ to itself.
  The Pontryagin dual group $\widehat{X} =: G$ is a subgroup of $\Q^d$
  which is torsion-free Abelian group of finite Pr\"ufer rank $d$,
  on which we have induced endomoprhisms $\widehat{\alpha}, \widehat{\beta}: G \to G$.
  Denote by $\widehat{\alpha}_\Q, \widehat{\beta}_\Q$ the extensions of $\widehat{\alpha}, \widehat{\beta}$
  to the divisible hull $G_\Q = \Q \otimes G \cong \Q^d$ and suppose that the endomorphisms
  $\widehat{\alpha}_\Q, \widehat{\beta}_\Q$ are simultaneously triangularizable.
  Let $\xi_1, \dots, \xi_d$ and $\eta_1, \dots, \eta_d$ denote the eigenvalues
  of $\widehat{\alpha}_\Q, \widehat{\beta}_\Q$ in a fixed algebraic closure of $\Q$,
  including multiplicities, ordered so that for $n \in \N$ the eigenvalues of
  $\widehat{\alpha}_\Q^n - \widehat{\beta}_\Q^n$ are $\xi_1^n - \eta_1^n, \dots, \xi_d^n - \eta_d^n$.
  Then the Reidemeister coincidence numbers $R(\widehat{\alpha}_\Q, \widehat{\beta}_\Q)$
  can be computed using \eqref{eq:Reidemeister-coincidence-product-eigenvalues}.

  \begin{enumerate}
    \item[\rm (i)] Assume that the primes $p$ that contribute non-trivial factors
      $\prod_{i \not \in I(p)} \lvert \xi_i^{\, n} - \eta_i^{\, n} \rvert_p$
      to the product on the far right-hand side of \eqref{eq:Reidemeister-coincidence-product-eigenvalues}
      form a finite subset $\mathbf{P} \subseteq \mathbb{P}$ and that
      $|\xi_i|_\infty \neq |\eta_i|_\infty$ for $1 \leq i \leq d$.
      Then the synchronization zeta function $S_{\alpha, \beta}(z)$  is either a rational function or it has a natural boundary on its radius of convergence.
      Furthermore, the latter occurs if and only if $|\xi_i|_p = |\eta_i|_p$ for some
      $p \in \mathbf{P}$ and $i \notin I(p)$.

    \item[\rm (ii)] With the assumption of finiteness of $\mathbf{P}$ and
      $|\xi_i|_\infty \neq |\eta_i|_\infty$ for $1 \leq i \leq d$
      the growth rate $S^\infty(\alpha, \beta)$ exists and is given by
      \begin{align} \label{eq:S-infty-formula}
      \begin{split}
        S^\infty(\alpha, \beta) &= S^\infty_+(\alpha, \beta) = S^\infty_-(\alpha, \beta) = \\
        &= \prod_{p \in \mathbf{P}} \prod_{i \in S^*(p)} \max \{ \lvert \xi_{i} \rvert_p,  |\eta_{i}|_p\} \cdot
             \prod_{p \in \mathbf{P}} \, \prod_{i \in S(p)\setminus S^*(p)}\lvert \eta_{i} \rvert_p  \cdot
             \prod_{i=1}^{d} \max\{|\xi_{i}|_\infty, |\eta_{i}|_\infty\},
      \end{split}
      \end{align}
      where $I(p)$ is as in {\rm Theorem \ref{thm:FelsKlop-3.4}}, for $p \in \mathbf{P}$ we write $S(p) = \{1,\dots,d\}\setminus I(p)$
      and $S^*(p) = \big\{i \in S(p) \mid \lvert \xi_{i} \rvert_p \neq |\eta_{i}|_p \big\}$
  \end{enumerate}
\end{thm}
\begin{proof}
  We can write
  \begin{align} \label{eq:synchronization=reidemeister-of-dual}
  \begin{split}
    \# S(\alpha^n, \beta^n) &= \# \Ker(\alpha^n - \beta^n)
    = \# \widehat{\Coker}(\widehat{\alpha}^n - \widehat{\beta}^n) = \\
    &= \# \Coker(\widehat{\alpha}^n - \widehat{\beta}^n) = R(\widehat{\alpha}^n, \widehat{\beta}^n)
  \end{split}
  \end{align}
  The endomoprhisms $\widehat{\alpha}, \widehat{\beta}$ fulfill assumptions of \cite[Theorem~4.3]{FelsKlop}
  and we can follow the logic of its proof to show (i).
  Let us use notation of Theorem \ref{thm:FelsKlop-3.4}.
  For $p \in \mathbf{P}$ we write
  \begin{align*}
    S(p) = \{1,\dots,d\}\setminus I(p), \quad S^*(p) = \big\{i \in S(p) \mid \lvert \xi_{i} \rvert_p \neq |\eta_{i}|_p \big\}
  \end{align*}
  and set
  \begin{align*}
     b := \prod_{p \in \mathbf{P}} \prod_{i \in S^*(p)} \max \{ \lvert \xi_{i} \rvert_p, |\eta_i|_p\}, \qquad
     \eta := \prod_{p \in \mathbf{P}} \, \prod_{i \in S(p) \setminus S^*(p)}\lvert \eta_{i} \rvert_p.
  \end{align*}
  We will show that
  \begin{align*}
    R(\widehat{\alpha}^n, \widehat{\beta}^n) = g(n) \cdot f(n)
  \end{align*}
  where
  \begin{gather*}
      g(n) := b^n\cdot\eta^n \cdot \prod_{i=1}^{d} \lvert \xi_{i}^{\, n} - \eta_{i}^{\, n}
      \rvert_\infty \\
      f(n) := \prod_{p \in \mathbf{P}} \, \prod_{i \in S(p)\setminus
        S^*(p)} \big\vert (\xi_{i} \eta_{i}^{\, -1})^n - 1 \big\vert_p.
  \end{gather*}
  Starting with \eqref{eq:Reidemeister-coincidence-product-eigenvalues} we have
  \begin{align}\label{eq:Reid}
  \begin{split}
    R(\widehat{\alpha}^n,\widehat{\beta}^n)
        & =  \prod_{i=1}^d \lvert \xi_i^{\, n} - \eta_i^{\, n} \rvert_\infty \cdot
          \prod_{p \in \mathbf{P}} \, \prod_{i  \in S(p)} \lvert \xi_i^{\, n} -
          \eta_i^{\, n} \rvert_p \\
        & = \prod_{i=1}^d \lvert \xi_i^{\, n} - \eta_i^{\, n}
          \rvert_\infty \cdot \prod_{p \in \mathbf{P}} \, \prod_{i \in S^*(p)}
          \max\{\lvert \xi_i \rvert_p^n, \lvert \eta_i \rvert_p^n\} \cdot
          \prod_{p \in \mathbf{P}} \, \prod_{i \in S(p)\setminus
          S^*(p)}\lvert \xi_i^{\, n} - \eta_i^{\, n} \rvert_p \\
        & = \prod_{i=1}^d \lvert \xi_i^{\, n} - \eta_i^{\, n}
          \rvert_\infty \cdot b^n \cdot \eta^n \cdot \prod_{p \in
          \mathbf{P}} \, \prod_{i \in S(p)\setminus S^*(p)} \big\vert (\xi_i
          \eta_i^{\, -1})^n - 1 \big\vert_p \\
        & = g(n) \cdot f(n).
  \end{split}
  \end{align}

  Now we open up the product $\prod_{i=1}^{d} |\xi_{i}^{\, n} - \eta_{i}^{\, n}|_\infty$.
  If at least one of $\xi_{i}, \eta_{i}$ is complex so that they are paired with
  $\xi_{j} = \overline{\xi_{i}}, \eta_{j} = \overline{\eta_{i}}$ for suitable $j \neq i$
  as discussed above, we see that
  \[
    |\xi_{i}^{\, n} - \eta_{i}^{\, n}|_\infty \cdot |\xi_{j}^{\, n} - \eta_{j}^{\, n}|_\infty
    = |\xi_{i}^{\, n} - \eta_{i}^{\, n}|_\infty^2
    = (\xi_{i}^{\, n} - \eta_{i}^{\, n}) \cdot (\overline{\xi_{j}}^{\, n} - \overline{\eta_{j}}^{\, n})
  \]
  If $\xi_{i}, \eta_{i}$ are both real eigenvalues, not paired up with another pair of eigenvalues,
  then we have
  \begin{gather*}
    |\xi_{i}^{\, n} - \eta_{i}^{\, n}|_\infty = \delta_{1,i}^n - \delta_{2,i}^n,\\
    \delta_{1,i} := \max \{|\xi_{i}|_\infty, |\eta_{i}|_\infty\}, \quad
    \delta_{2,i} := \frac{\xi_{i} \cdot \eta_{i}}{\delta_{1,i}}\\
  \end{gather*}
  We can now expand the product $\prod_{i=1}^{d} |\xi_{i}^{\, n} - \eta_{i}^{\, n}|_\infty$ using
  an appropriate symmetric polynomial to express
  \begin{align} \label{eq:gn-as-sum}
    g(n) = \sum_{j \in J} c_j w_j^n
  \end{align}
  where $J$ is a finite index set, $c_j = \pm 1$ and $w_j \in \C \setminus \{0\}$.

  Consequently, the synchronization zeta function can be written as
  \begin{align*}
    S_{\alpha, \beta}(z) = \exp \left( \sum_{j \in J} c_j \sum_{n=1}^\infty \frac{f(n) \cdot (w_j z)^n}{n} \right).
  \end{align*}
  If $S(p) = S^*(p)$ for all $p \in \mathbf{P}$ then $f(n) = 1$ for all $n$
  and straightforward computation leads us to
  \begin{align*}
     S_{\alpha, \beta}(z) = \prod_{j \in J} (1 - w_j z)^{-c_j}
  \end{align*}
  which is a rational function.

  Now suppose that for some $p \in \mathbf{P}$ we have $S(p) \neq S^*(p)$.
  By Lemma \ref{thm:logarithmic-derivative}, it suffices to show that
  \begin{align*}
    Z_{\widehat{\alpha}, \widehat{\beta}}(z) := \sum_{j \in J} c_j \sum_{n=1}^\infty f(n) \cdot (w_j z)^n
  \end{align*}
  has a natural boundary at its radius of convergence.
  By Lemma \ref{thm:BMW-lemma17} we have $\limsup_{n\to\infty} f(n)^{1/n} = 1$.
  Hence, for each $j \in J$, the series $\sum_{n\geq 1} f(n) \cdot (w_j s)^n$
  has the radius of convergence $|w_j|_\infty^{-1}$.
  As  $|\xi_i|_\infty \neq |\eta_i|_\infty$ for $1 \leq i \leq d$, there is a dominant term $w_m$ in the expansion
  \eqref{eq:gn-as-sum}   for which
  \begin{align} \label{eq:domina}
    |w_m|_\infty =  b \cdot \eta \cdot
     \prod_{i=1}^{d} \max \big\{ |\xi_{i}|_\infty, |\eta_{i}|_\infty \big\} > |w_j|_\infty, \quad j \neq m.
  \end{align}
  Thus it suffices to show that the series $\sum_{n=1} f(n) \cdot (w_m z)^n$ has  its circle  of convergence as a natural boundary. This is the case, because the series $\sum_{n=1} f(n) \cdot (z)^n$ has the unit circle as a natural boundary by Lemma \ref{thm:BMW-lemma17}.

  To prove (ii) observe that the assumption of finiteness of $\mathbf{P}$ and inequality
  $|\xi_i|_\infty \neq |\eta_i|_\infty$ for $1 \leq i \leq d$ is the same as in (i).
  From what has been said above we have $\limsup_{n \to \infty} f(n)^{1/n} = 1$.
  From Lemma \ref{thm:FelSlo-lemma2.4} we have also $\liminf_{n \to \infty} f(n)^{1/n} = 1$.
  Therefore $\lim_{n \to \infty} f(n)^{1/n} = 1$.

  Using formulae \eqref{eq:synchronization=reidemeister-of-dual} and \eqref{eq:Reid}
  for the Reidemeister coincidence numbers $R(\widehat{\alpha}^n,\widehat{\beta}^n)$,
  and the formula \eqref{eq:domina} for dominant term $w_m$ we get
  \begin{align}
    S^\infty(\alpha, \beta) = R^\infty(\widehat{\alpha}, \widehat{\beta})  = |w_m|_\infty
  \end{align}
  which is exactly \eqref{eq:S-infty-formula}.
\end{proof}

\begin{col}\label{torus}
  Consider the $d$-dimensional torus $X = \T^d$, $d \geq 1$ and a synchronously tame pair of endomorphisms
  $\alpha, \beta: \T^d \to \T^d$.
  The Pontryagin dual $\widehat{\T^d} = \Z^d$.
  Using notation of {\rm Theorem~\ref{thm:synchronization-dichotomy}}
  let $\widehat{\alpha}, \widehat{\beta}: \Z^d \to \Z^d$ denote the endomorphisms induced
  on the finite-dimensional dual group, assume that the extensions
  $\widehat{\alpha}_\Q, \widehat{\beta}_\Q$ are simultaneously triangularizable
  and their ordered eigenvalues are $\xi_1, \ldots, \xi_d$ and $\eta_1, \ldots, \eta_d$, respectively.

  If $|\xi_i|_\infty \neq |\eta_i|_\infty$ for $1 \leq i \leq d$ then the synchronization zeta function
  $S_{\alpha, \beta}(z)$ is rational and the growth rate $S^\infty(\alpha, \beta)$ exists and is given by
  \begin{equation*}
    S^\infty(\alpha, \beta) \ = \ \prod_{i=1}^d \max \{ |\xi_i|_\infty, |\eta_i|_\infty \}.
  \end{equation*}
\end{col}
\begin{proof}
  It is convenient to use \eqref{eq:synchronization=reidemeister-of-dual} as follows:
  \begin{align*}
  \begin{split}
    \# S(\alpha^n, \beta^n) &= \# \Ker(\alpha^n - \beta^n)
    = \# \widehat{\Coker}(\widehat{\alpha}^n - \widehat{\beta}^n) = \\
    &= \# \Coker(\widehat{\alpha}^n - \widehat{\beta}^n)
    = \left| \det \left( \widehat{\alpha}_\Q^n - \widehat{\beta}_\Q^n \right) \right|_\infty \\
    &= \prod_{i=1}^d | \xi_i^n - \eta_i^n |_\infty
  \end{split}
  \end{align*}
  We open up the absolute values in the same way as in Theorem~\ref{thm:synchronization-dichotomy}
  to get \eqref{eq:gn-as-sum} and we obtain
  \begin{equation} \label{eq:synchronization-on-Td-sum}
    \# S(\alpha^n, \beta^n) = \sum_{j \in J} c_j w_j^n
  \end{equation}
  where $J$ is a finite set, $c_j = \pm 1$ and $w_j \in \C \setminus \{0\}$.
  Straightforward computation (cf. Example~\ref{ex:synchronization-on-T2}) leads us to
  \begin{equation*}
    S_{\alpha, \beta}(z) \ = \ \prod_{j \in J} (1 - w_j)^{-c_j}
  \end{equation*}
  which is a rational function.

  As $|\xi_i|_\infty \neq |\eta_i|_\infty$ for all $1 \leq i \leq d$, there is a dominant term $w_m$
  in the expansion \eqref{eq:synchronization-on-Td-sum}, for which
  \begin{equation*}
    |w_m|_\infty = \prod_{i=1}^d \max \{ |\xi_i|_\infty, |\eta_i|_\infty \}
  \end{equation*}
  and $|w_m|_\infty > |w_j|_\infty$ for $j \neq m$.
  Using \eqref{eq:synchronization-on-Td-sum} we get $S^\infty(\alpha, \beta) = |w_m|_\infty$.
\end{proof}

\begin{rmk}  The result of Corollary  \ref{torus} holds for any two {\bf continuous} self-maps
$\alpha, \beta: \T^d \to \T^d$ in the form of inequality
\begin{equation*}
    S^\infty(\alpha, \beta) \geq \prod_{i=1}^d \max \{ |\xi_i|_\infty, |\eta_i|_\infty \},
  \end{equation*}

where $\xi_i, \eta_i $ are the eigenvalues of the linearizations of these maps, i.e. endomorphisms of
$\T^d $ homotopic to $\alpha, \beta$ respectively, given by the induced maps on $H_1(\T^d; Z )$.
Indeed, for torus maps we have the formula for the Nielsen coincidence number given by J. Jezierski
in \cite{Jez2} from which we get $ N(\alpha^n, \beta^n)  = \prod_{i=1}^d | \xi_i^n - \eta_i^n |_\infty $.
But $ \# S(\alpha^n, \beta^n) \geq   N(\alpha^n, \beta^n) $.
\end{rmk}

\subsection{Examples}

\begin{ex} \label{ex:synchronization-on-T1}
  Let $X = S^1$ be the unit circle on the complex plane.
   Consider maps $\alpha(z) := z^{d_\alpha}, \beta(z) := z^{d_\beta}, d_\alpha, d_\beta \in \Z$, $|z| = 1$, of $S^1$ to itself.
  If $d_\alpha^n = d_\beta^n$ then $S(\alpha^n, \beta^n) = S^1$, hence $\# S(\alpha^n, \beta^n) = \infty$.
  Now assume that $d_\alpha^n \neq d_\beta^n$. We can write
  \begin{align*}
    z \in S(\alpha^n, \beta^n) \ &\Leftrightarrow \ z^{d_\alpha^n} = z^{d_\beta^n}
      \ \Leftrightarrow \ z^{d_\beta^n} \left( z^{d_\alpha^n - d_\beta^n} - 1 \right) = 0
      \ \Leftrightarrow \  z^{d_\alpha^n - d_\beta^n} - 1 = 0.
  \end{align*}
  We conclude that  $  \# S(\alpha^n, \beta^n) =|d_\alpha^n - d_\beta^n|$ .
  Therefore
  \begin{align*}
    \# S(\alpha^n, \beta^n) = \begin{cases}
      |d_\alpha^n - d_\beta^n|, & d_\alpha^n \neq d_\beta^n \\
      \infty, & d_\alpha^n = d_\beta^n
    \end{cases}.
  \end{align*}
  Consequently, the pair $\alpha, \beta$  is  synchronously tame precisely when $|d_\alpha| \neq |d_\beta|$.
  Now we are able to compute the synchronization zeta function assuming the pair $\alpha, \beta$ is  synchronously tame.
  Denote $d_1 := \max \{ |d_\alpha|, |d_\beta| \}, d_2 := d_\alpha d_\beta / d_1$.
  Then  $|d_\alpha^n - d_\beta^n| = d_1^n - d_2^n$.
  We have
  \begin{align*}
    S_{\alpha, \beta}(z) &= \exp \left( \sum_{n=1}^\infty \frac{|d_\alpha^n - d_\beta^n|}{n} z^n \right) \\
    &= \exp \left( \sum_{n=1}^\infty \frac{(d_1 z)^n}{n} - \sum_{n=1}^\infty \frac{(d_2 z)^n}{n} \right) \\
    &= \exp \big( -\ln (1 - d_1 z) + \ln (1 - d_2 z) \big) \\
    &= \frac{1 - d_2 z}{1 - d_1 z}.
  \end{align*}
  Let us compute the growth rate
  \begin{align*}
    S^\infty(\alpha, \beta) = \lim_{n \to \infty} \sqrt[n]{d_1^n - d_2^n}
      = \lim_{n \to \infty} \sqrt[n]{d_1^n} \cdot \lim_{n \to \infty} \sqrt[n]{1 - \left( \frac{d_2}{d_1} \right)^n}
      = d_1.
  \end{align*}
\end{ex}

\begin{ex} \label{ex:synchronization-on-T2}
  Consider $2$-dimensional torus $\T^2$ and its two hyperbolic automorphisms
  \begin{align*}
    \alpha = \begin{pmatrix} \lambda & 0 \\ 0 & \lambda^{-1} \end{pmatrix}, \quad
    \beta = \begin{pmatrix} \mu & 0 \\ 0 & \mu^{-1} \end{pmatrix}, \quad
    \lambda > \mu > 1, \quad \lambda, \mu \in \R.
  \end{align*}
  For any $x \in \T^2$ we have
  \begin{align*}
    \alpha^n x = \beta^n x \ \Leftrightarrow \ (\alpha^n - \beta^n) x = 0
  \end{align*}
  therefore
  \begin{align*}
    \# S(\alpha^n, \beta^n) &= \# \Ker(\alpha^n - \beta^n) = \# \Ker(\beta^{-n}\alpha^n - \Id) \\
    &= |\det(\beta^{-n}\alpha^n - \Id)| = |\det(\alpha^n - \beta^n)|.
  \end{align*}
  In terms of eigenvalues we get
  \begin{align} \label{eq:toral-auto-synchronization-eigen}
  \begin{split}
    \# S(\alpha^n, \beta^n)
      &= \left| \det \begin{pmatrix} \lambda^n - \mu^n & 0 \\ 0 & \lambda^{-n} - \mu^{-n} \end{pmatrix} \right| \\
      &= \left( \lambda^n - \mu^n \right) \left( \mu^{-n} - \lambda^{-n} \right) \\
      &= \left( \frac{\lambda}{\mu} \right)^n - 2 + \left( \frac{\mu}{\lambda} \right)^n.
  \end{split}
  \end{align}
  The synchronization zeta function becomes rational:
  \begin{align*}
    S_{\alpha, \beta}(z) \ &= \ \exp \left( \sum_{n=1}^\infty \frac{-2 + \left( \frac{\lambda}{\mu} \right)^n + \left( \frac{\mu}{\lambda} \right)^n}{n} z^n \right) \\
    &= \ \exp \left( -2 \sum_{n=1}^\infty \frac{z^n}{n} + \sum_{n=1}^\infty \frac{\left( \frac{\lambda}{\mu} z \right)^n}{n}
      + \sum_{n=1}^\infty \frac{\left( \frac{\mu}{\lambda} z \right)^n}{n} \right) \\
    &= \ \exp \left( 2 \ln (1-z) - \ln \left(1 - \frac{\lambda}{\mu} z\right) - \ln\left(1 - \frac{\mu}{\lambda} z\right) \right) \\
    &= \ \frac{(1-z)^2}{\left(1 - \frac{\lambda}{\mu} z\right) \left(1 - \frac{\mu}{\lambda} z\right)}.
  \end{align*}
  Using \eqref{eq:toral-auto-synchronization-eigen} we can also compute the growth rate.
  Because, by assumption, $\lambda > \mu > 1$, we have
  \begin{align*}
    S^\infty(\alpha, \beta) \ &= \ \lim_{n \to \infty} \sqrt[n]{ -2 + \left( \frac{\lambda}{\mu} \right)^n + \left( \frac{\mu}{\lambda} \right)^n} \ = \ \frac{\lambda}{\mu} \ .
  \end{align*}
\end{ex}

To illustrate emergence of the natural boundary for the synchronization zeta function
we use the following example which was previously analyzed for
Reidemeister coincidence zeta function in \cite[Paragraph~4.3]{FelZie}
and for Artin-Mazur zeta function in \cite[Proposition~2]{EvStanWard}.

\begin{ex}
  Consider the group $G = \Z[\frac{1}{3}]$ and its two endomorphisms
  $\varphi: g \mapsto 6g$ and $\psi: g \mapsto 3g$.
  The unitary dual $\widehat{G}$ is the $1$-dimensional solenoid on which we have two induced
  endomorphisms $\widehat{\varphi}, \widehat{\psi}$, respectively.

  Now by {\rm Lemma \ref{thm:FelHil-dual}} and the formula \eqref{eq:Reidemeister-coincidence-product-eigenvalues}
  in {\rm Theorem \ref{thm:FelsKlop-3.4}}
  \begin{align*}
    \# S(\widehat{\varphi}^n, \widehat{\psi}^n)
    &= \# \Ker(\widehat{\varphi}^n - \widehat{\psi}^n)
    = \# \widehat{\Coker}(\varphi^n - \psi^n) \\
    &= \# \Coker(\varphi^n - \psi^n) = R(\varphi^n, \psi^n) \\
    &= |6^n - 3^n|_\infty \cdot |6^n - 3^n|_3 = |2^n - 1|_\infty \cdot |2^n - 1|_3
  \end{align*}
  so that the zeta
  \begin{align*}
    S_{\widehat{\varphi}, \widehat{\psi}}(z) = \exp \left( \sum_{n=1}^\infty \frac{|2^n-1|_\infty \cdot |2^n - 1|_3}{n} z^n \right).
  \end{align*}
  Now we follow the method and the calculations of \cite[Lemma~4.1]{EvStanWard}.
  Let us focus on the power series
  \begin{align*}
    \xi(z) := \sum_{n=1}^\infty \frac{z^n}{n} |2^n-1|_\infty \cdot |2^n-1|_3
  \end{align*}
  Observe that
  \begin{align*}
    |2^n - 1|_3 &= |(3-1)^n - 1|_3 \\
    &= |3^n - n3^{n-1} + \dots + (-1)^{n-1}3n + (-1)^n - 1|_3 \\
    &= |3\left(3^{n-1} - n3^{n-2} + \dots + (-1)^{n-1}n \right) + (-1)^n - 1|_3 \\
    &= \begin{cases}
      \frac{1}{3}|n|_3, & \text{$n$ {\rm even}} \\
      1, & \text{$n$ {\rm odd}}
    \end{cases},
  \end{align*}
  in particular $|4^n-1|_3 = |2^{2n}-1|_3 = \frac{1}{3}|2n|_3 = \frac{1}{3}|n|_3$.
  It makes sense now to write
  \begin{align} \label{eq:example-synchronization-zeta-xi}
  \begin{split}
    \xi(z) \ &= \ \sum_{n=0}^\infty \frac{z^{2n+1}}{2n+1}(2^{2n+1}-1) \
      + \ \sum_{n=1}^\infty \frac{z^{2n}}{2n} (2^{2n}-1)|2^{2n}-1|_3 \\
    &= \ \log \left( \frac{1-z}{1-2z} \right) - \frac{1}{2} \log \left( \frac{1-z^2}{1-4z^2} \right) \
      + \ \sum_{n=1}^\infty \frac{z^{2n}}{2n} (2^{2n}-1)|2^{2n}-1|_3.
  \end{split}
  \end{align}

  Denote the last term by $\frac{1}{6}\xi_1(z)$. We have
  \begin{align*}
    \xi_1(z) \ = \ 3\sum_{n=1}^\infty \frac{z^{2n}}{n}(4^n-1)|4^n-1|_3 = \sum_{n=1}^\infty \frac{z^{2n}}{n}(4^n-1)|n|_3.
  \end{align*}
  We will show that $\xi_1(z)$ has infinitely many logarithmic singularities on the circle $|z| = 1/2$.

  We are going to use the notation $3^a \Divides n$ to denote $3^a \divides n$ and $3^{a+1} \not\divides n$.
  It is clear that $3^a \Divides n$ is equivalent to $|n|_3 = 3^{-a}$.
  Denote $\eta_j^{(a)}(z) = \sum_{3^j \Divides n} \frac{z^{2n}}{n}(a^n-1)$.
  We can split up $\xi_1$ in the following way
  \begin{align*}
    \xi_1(z) &= \sum_{j=0}^\infty \frac{1}{3^j} \sum_{3^j \Divides n} \frac{z^{2n}}{n}(4^n-1) \\
    &= \sum_{j=0}^\infty \frac{1}{3^j} \eta_j^{(4)}(z)
  \end{align*}
  Observe that
  \begin{align*}
    \eta_0^{(a)}(z) &= \sum_{3^j \Divides n} \frac{z^{2n}}{n}(a^n-1) \\
    &= \sum_{n=1}^\infty \frac{z^{2n}}{n}(a^n-1) - \sum_{n=1}^\infty \frac{z^{6n}}{3n}(a^{3n} - 1) \\
    &= \log \left( \frac{1-z^2}{1-az^2} \right) - \frac{1}{3} \log \left( \frac{1-z^6}{1-a^3z^6} \right).
  \end{align*}
  Moreover
  \begin{align*}
    \eta_1^{(4)}(z) = \sum_{3^1 \Divides n} \frac{z^{2n}}{n} (4^n-1) = \sum_{3^0 \Divides n} \frac{z^{6n}}{3n}(4^{3n}-1)
    = \frac{1}{3}\eta_0^{(4^3)}(z^3)
  \end{align*}
  and similarly
  \begin{align*}
    \eta_2^{(4)}(z) = \frac{1}{9} \eta_0^{(4^9)}(z^9)
  \end{align*}
  and so on.
  Therefore the power series
  \begin{align*}
    \xi_1(z) = \log \left( \frac{1-z^2}{1-(2z)^2} \right)
      + 2 \sum_{j=1}^\infty \frac{1}{9^j} \log \left( \frac{1-(2z)^{2 \cdot 3^j}}{1-z^{2 \cdot 3^j}} \right).
  \end{align*}
  Returning via \eqref{eq:example-synchronization-zeta-xi} to the synchronization zeta function
  we have its module
  \begin{align*}
    \left|S_{\widehat{\varphi},\widehat{\psi}}(z)\right|_\infty \ = \ \left| \frac{1-z}{1-2z} \right|_\infty
      \cdot \left| \frac{1-(2z)^2}{1-z^2} \right|_\infty^{1/2}
      \cdot \left| \frac{1-z^2}{1-(2z)^2} \right|_\infty^{1/6}
      \cdot \prod_{j=1}^\infty \left| \frac{1-(2z)^{2\cdot3^j}}{1-z^{2\cdot 3^j}} \right|_\infty^{1/(3\cdot 9^j)}.
  \end{align*}
  Note that if $2z$ is $(2\cdot3^j)$-th root of $1$, $j=1,2,\dots$, then $S_{\widehat{\varphi},\widehat{\psi}}(z) = 0$,
  therefore $|z| = 1/2$ is a natural boundary.
\end{ex}

\subsection{Rationality of synchronization zeta function on the dual space of virtually polycyclic groups}
We cite the definition of a virtually polycyclic group from \cite{Dere}.
A group $G$ is \emph{polycyclic} if it has a subnormal series
\begin{align*}
  \{ e \} = G_0 \vartriangleleft G_1 \vartriangleleft \dots \vartriangleleft G_l = G
\end{align*}
with $G_i$ normal in $G_{i+1}$ and the quotient $G_{i+1}/G_i$ cyclic.
A group $G$ is called \emph{virtually polycyclic} if it has a subgroup of finite index which is polycyclic.
Every finitely generated nilpotent group is automatically polycyclic.

\begin{thm}[\protect{\cite[Theorem 7.8]{FelTro}}, Twisted Burnside-Frobenius Theorem] \label{thm:FelTro-7.8}
  Let $G$ be a virtually polycyclic group and $\varphi: G \to G$ its automorphism.
  Then the Reidemeister number $R(\varphi)$ is equal to the number of fixed points
  of the map $\widehat{\varphi}$ induced on the finite-dimensional part $\widehat{G}_{fin}$
  of the unitary dual of $G$.
\end{thm}

\begin{thm}[\protect{\cite[Theorem 3.12]{Dere}}] \label{thm:Dere-3.12}
  Let $\varphi: \Gamma \to \Gamma$ be a monomorphism of a virtually polycyclic group.
  If $R(\varphi^n) < \infty$ for all integers $n>0$, then the group $\Gamma$ must be virtually nilpotent.
  In particular, the Reidemeister zeta function $R_\varphi(z)$ of $\varphi$ is rational
  if the group $\Gamma$ is additionally torsion-free.
\end{thm}

\begin{thm} \label{thm:synchronization-on-dual-of-virtually-polycyclic=reidemeister}
  Let $G$ be a torsion-free virtually polycyclic group
  and $\varphi, \psi: G \to G$ a~synchronously tame pair of commuting automorphisms.
  Denote the space of unitary irreducible representations of $G$ by $\widehat{G}$,
  its subspace of finite dimensional representations by $\widehat{G}_{fin}$
  and dual maps induced by $\varphi, \psi$ on the subspace $\widehat{G}_{fin}$ by $\widehat{\varphi}_{fin}, \widehat{\psi}_{fin}$ respectively.

  Then $\# S(\widehat{\varphi}_{fin}^n, \widehat{\psi}_{fin}^n) = R(\varphi^n, \psi^n)$ for all $n \in \N$
  and the synchronization zeta function $S_{\widehat{\varphi}_{fin}, \widehat{\psi}_{fin}}(z) = R_{\varphi, \psi}(z)$
  is rational.
\end{thm}
\begin{proof}
  For bijections $\widehat{\varphi}_{fin}, \widehat{\psi}_{fin}$
  \begin{align*}
    \rho \in S(\widehat{\varphi}_{fin}^n, \widehat{\psi}_{fin}^n) \ &\Leftrightarrow \ \widehat{\varphi}_{fin}^n(\rho) = \widehat{\psi}_{fin}^n(\rho)
      \ \Leftrightarrow \ \left( \widehat{\psi}_{fin}^{-n} \circ \widehat{\varphi}_{fin}^n \right)(\rho) = \rho  \\
    &\Leftrightarrow \ \left( \widehat{(\varphi^n \circ \psi^{-n})}_{fin} \right)(\rho) = \rho
      \ \Leftrightarrow \ \rho \in \Fix \left(\widehat{(\varphi^n \circ \psi^{-n})}_{fin} \right)
  \end{align*}
  therefore, by commutativity of $\widehat{\varphi}_{fin}, \widehat{\psi}_{fin}$
  and by Theorem~\ref{thm:FelTro-7.8}
  \begin{align*}
    \# S(\widehat{\varphi}_{fin}^n, \widehat{\psi}_{fin}^n)
      &= \# \Fix \left( \widehat{(\varphi^n \circ \psi^{-n})}_{fin} \right)
      = R(\psi^{-n} \circ \varphi^n) = R(\varphi^n, \psi^n).
  \end{align*}

  Observe that $\psi^{-n} \circ \varphi^n$ is an automorphism, so in particular it is a monomoprhism,
  therefore by Theorem~\ref{thm:Dere-3.12} the~group $G$ must be virtually nilpotent
  and because $G$ is additionally by hypothesis torsion-free, the zeta function
  \begin{align*}
    S_{\widehat{\varphi}_{fin}, \widehat{\psi}_{fin}}(z) = R_{\varphi, \psi}(z) = R_{\psi^{-1}\varphi}(z)
  \end{align*}
  is rational.
\end{proof}

\subsection{Rationality of synchronization zeta function for Axiom A diffeomoprhisms} \label{sec:axiom-A}

We recall definitions related to the Smale's Axiom A as in \cite{Sma67}.
Let $M$ be a compact manifold and $f \in \Diff(M)$ be a diffeomorphism
such that the set $\Fix(f)$ of its fixed points is finite.
A point $x \in M$ is \emph{wandering} if there is an open neighborhood $U$ of $x$
such that $\bigcup_{|m|>0} f^m(U) \cap U = \varnothing$.
A point is \emph{nonwandering} if it is not wandering.
Nonwandering points form a closed invariant set which is denoted by $\Omega(f)$.

A linear automorphism $u: V \to V$ of a finite dimensional vector space is \emph{hyperbolic}
if all its eigenvalues $\lambda_i$ satisfy $|\lambda_i| \neq 1$,
\emph{contracting} if all its eigenvalues satisfy $|\lambda_i| < 1$,
and \emph{expanding} if all its eigenvalues satisfy $|\lambda_i| > 1$.
A fixed point $p \in \Fix(f)$ of a diffeomorphism is \emph{hyperbolic}
if the derivative of $f$ at $p$, $Df(p): T_p(M) \to T_p(M)$ is hyperbolic
as a linear automoprhism.

Let $\Lambda \subset M$ be closed subset invariant under $f$.
The set $\Lambda$ is \emph{hyperbolic} (or has a \emph{hyperbolic structure})
if the tangent bundle of $M$ restricted to $\Lambda$, $T_\Lambda(M)$ has an invariant continuos
splitting $T_\Lambda(M) = E^u + E^s$ such that $Df: E^s \to E^s$ is contracting
and $Df: E^u \to E^u$ expanding.
The sets $\Omega(f^k)$ for $k \in \N$ are examples of hyperbolic sets.

We say that a diffeomorphism $f: M \to M$ of a compact manifold \emph{satisfies Axiom A}
if: (a) the nonwandering set $\Omega(f)$ is hyperbolic and
(b) the periodic points of $f$ are dense in $\Omega(f)$.
If, additionally, $M = \Omega(f)$, then $f$ is an \emph{Anosov diffeomoprhism}.

\begin{thm}[{\protect \cite{Man71}}] \label{thm:Manning}
  Let $f: M \to M$ be a $C^1$ diffeomorphism of a $C^\infty$ compact manifold $M$ without boundary.
  If $f$ satisfies Axiom A then the Artin-Mazur zeta function
  is rational.
\end{thm}

Observe that when $f, g$ commute then
\begin{equation*}
  x \in \Fix \left( (g^{-1}f)^n \right) \ \Leftrightarrow \ x \in S(f^n, g^n) \ \Leftrightarrow \
    x \in \Fix \left( (f^{-1}g)^n \right)
\end{equation*}
for all $n \in \N$
and the corresponding Artin-Mazur and synchronization zeta functions are equal:
\begin{equation*}
  \zeta_{g^{-1}f}(z) \ = \ S_{f,g}(z) \ = \ \zeta_{f^{-1}g}(z).
\end{equation*}
Then Theorem \ref{thm:Manning} immediately implies the following
\begin{thm} \label{thm:Manning2}
  Let $f,g: M \to M$ be $C^1$
  commuting diffeomorphisms of a $C^\infty$ compact manifold without boundary.
  If $(g^{-1}f)$ or $(f^{-1}g)$ satisfies Axiom A, then the synchronization zeta function
  $S_{f,g}(z)$ is rational.
\end{thm}

\begin{col}
  If $(g^{-1}f)$ or $(f^{-1}g)$ is Anosov diffeomorphism then the synchronization zeta function
  $S_{f,g}(z)$ is rational.
\end{col}

\subsection{Rationality of synchronization zeta function for pseudo-Anosov homeomorphisms}

A \emph{measured foliation} $(\mathcal{F}, \mu)$ on a closed compact manifold $M$
is a foliation $\mathcal{F}$ together with a transverse measure $\mu$ invariant under holonomy.
Assume additionally that $M$ is a surface, i.e. $\dim M = 2$.
A~homeomorphism $f: M \to M$ is \emph{pseudo-Anosov} if there are two
mutually transverse, invariant under $f$, measured singular foliations $(\mathcal{F}^s, \mu^s)$
and $(\mathcal{F}^u, \mu^u)$, and a $\lambda > 1$, such that
\begin{align*}
  f(\mathcal{F}^s, \mu^s) &= (\mathcal{F}^s, \frac{1}{\lambda}\mu^s) \\
  f(\mathcal{F}^u, \mu^u) &= (\mathcal{F}^u, \lambda \mu^u).
\end{align*}
We call $\mathcal{F}^s, \mathcal{F}^u$ the \emph{stable} and \emph{unstable} foliations, respectively.
Meaning of the equalities above is that $f$ is contracting on the leaves of the stable foliation
and expanding on the leaves of the unstable foliation with the rate $\lambda^{-1}$ and $\lambda$, respectively.

A subset $R \subset M$ is an \emph{$(\mathcal{F}^s, \mathcal{F}^u)$-rectangle}, or \emph{birectangle},
if there exists an immersion $\iota: [0,1] \times [0,1] \to M$ whose image is $R$ and
\begin{enumerate}
\item[(a)] $f|_{(0,1) \times (0,1)}$ is an embedding;
\item[(b)] $\forall t \in [0,1]$ the image $\iota(\{t\} \times [0,1])$ is included in a finite union
           of leaves and singularities of $\mathcal{F}^s$;
\item[(c)] $\forall t \in [0,1]$ the image $\iota([0,1] \times \{t\})$ is included in a finite union
           of leaves and singularities of $\mathcal{F}^u$.
\end{enumerate}
A set of the form $\iota(\{t\} \times [0,1])$ or $\iota([0,1] \times \{t\})$ is an \emph{$\mathcal{F}^s$-fiber}
or \emph{$\mathcal{F}^u$-fiber}, respectively.
A birectangle is \emph{good} if $\iota$ is an embedding.
By $\mathcal{F}(x,R)$ we are going to denote the part of the $\mathcal{F}$-fiber
going through $x$ and contained in the birectangle $R$.

A \emph{Markov partition} for a pseudo-Anosov homeomorphism $f: M \to M$ is a collection of birectangles
$R = \{R_1, \ldots, R_k\}$ such that
\begin{enumerate}
\item[(a)] $\bigcup_{i=1}^k R_i \ = \ M$;
\item[(b)] $R_i$ is a good rectangle;
\item[(c)] $\Int(R_i) \cap \Int(R_j) = \emptyset$ for $i \neq j$;
\item[(d)] if $x \in \Int(R_i)$ and $f(x) \in \Int(R_j)$, then
\begin{equation*}
  f(\mathcal{F}^s(x, R_i)) \subset \mathcal{F}^s(f(x), R_j) \quad \text{and} \quad
  f^{-1}(\mathcal{F}^u(f(x), R_j)) \subset \mathcal{F}^u(x, R_i);
\end{equation*}
\item[(e)] if $x \in \Int(R_i)$ and $f(x) \in \Int(R_j)$, then
\begin{equation*}
  f(\mathcal{F}^u(x, R_i)) \cap R_j = \mathcal{F}^u(f(x), R_j) \quad \text{and} \quad
  f^{-1}(\mathcal{F}^s(f(x), R_j)) \cap R_i = \mathcal{F}^s(x, R_i);
\end{equation*}
\end{enumerate}

\begin{thm}[\protect{\cite[Proposition 10.17]{FatLauPoe2012}}]
  A pseudo-Anosov diffeomorphism has a Markov partition.
\end{thm}

Given a Markov partition $R = \{R_1, \ldots, R_k\}$, we construct the subshift of finite type
$\sigma_A: \Omega_A \to \Omega_A$ letting $A$ be the $k \times k$ matrix defined by
\begin{align*}
  a_{ij} = \begin{cases}
    1, &f(\Int R_i) \cap \Int R_j \neq \emptyset \\
    0, &\text{otherwise}
  \end{cases}.
\end{align*}
If $(\ldots,b_{-1},b_0,b_1,\ldots) \in \Omega_A$ then $\bigcap_{i \in \Z} f^{-i}(R_{b_i})$
consists of a single point of $M$.
This gives us the projection map $p: \Omega_A \to M$
satisfying $p \circ \sigma_A = f \circ p$.

Because birectangles $R_i$ can meet on their boundaries, the map $p$ is not injective in general.
Therefore, we cannot simply count fixed points of $f$ by counting fixed points of $\sigma$.
To overcome this difficulties we can use Manning's arguments from \cite{Man71}
(for details regarding pseudo-Anosov homeomorphisms see \cite[Lemma~27]{FelshB}).
In this way we obtain finitely many subshifts of finite type $\sigma_{A_i}, i=0,1,\ldots,m$,
where $A_i$ is the transition matrix for $\sigma_{A_i}$,
such that
\begin{equation} \label{eq:pseudo-Anosov-trace-formula}
  \#\Fix f^n \ = \ \sum_{i=0}^m \varepsilon_i \cdot \#\Fix(\sigma_{A_i}^n) \ = \ \sum_{i=0}^m \varepsilon_i \cdot \tr A_i^n,
\end{equation}
where all $\varepsilon_i \in \{-1,1\}$.
The second equality is the result of Theorem~\ref{thm:trace-formula-for-subshifts}.
This leads us to the following
\begin{thm} \label{thm:pseudo-Anosov-AM-rational}
  Let $f: M \to M$ be a pseudo-Anosov homeomorphism of a closed compact surface.
  The Artin-Mazur zeta function is rational:
  \begin{equation*}
    \zeta_f(z) \ = \ \prod_{i=0}^m \det(1 - A_i z)^{-\varepsilon_i},
  \end{equation*}
  where $\epsilon_i \in \{-1,1\}$.
\end{thm}
\begin{proof}
  We are going to use the following identities
  \begin{align*}
    \sum_{n=1}^\infty \frac{\tr A^n}{n} z^n = -\tr \big( \log (1 - Az) \big), \quad
    \exp(\tr A) = \det(\exp A).
  \end{align*}
  Direct computation using \eqref{eq:pseudo-Anosov-trace-formula} provides us with the result:
  \begin{align*}
    \zeta_f(z) \ &= \ \exp \left( \sum_{n=1}^\infty \frac{\#\Fix(f^n)}{n} z^n \right)
                  \ = \ \exp \left( \sum_{n=1}^\infty \frac{\sum_{i=0}^m \epsilon_i \cdot \tr A_i^n}{n} z^n \right) \\
                 &= \ \exp \left( \sum_{i=0}^m \epsilon_i \sum_{n=1}^\infty \frac{\tr A_i^n}{n} z^n \right)
                  \ = \ \prod_{i=0}^m \exp \left( \sum_{n=1}^\infty \frac{\tr A_i^n}{n} z^n \right)^{\epsilon_i} \\
                 &= \ \prod_{i=0}^m \exp \Big( -\tr \big(\log(1 - A_i z) \big) \Big)^{\epsilon_i} \\
                 &= \ \prod_{i=0}^m \det \Big( \exp \big(\log(1 - A_i z) \big) \Big)^{-\epsilon_i}
                  \ = \ \prod_{i=0}^m \det(1 - A_i z)^{-\varepsilon_i}.
  \end{align*}
\end{proof}

Similarly as for Axiom A diffeomorphisms, we can use this result to state
\begin{thm}\label{thm:pseudo-Anosov-AM-rational2}
  Let $f,g: M \to M$ be commuting homeomorphisms of a closed compact surface $M$.
  If $(g^{-1}f)$ or $(f^{-1}g)$ is pseudo-Anosov then the synchronization zeta function $S_{f,g}(z)$
  is rational.
\end{thm}
\begin{proof}
  Observe that when $f, g$ commute then
  \begin{equation*}
    x \in \Fix \left( (g^{-1}f)^n \right) \ \Leftrightarrow \ x \in S(f^n, g^n) \ \Leftrightarrow \
      x \in \Fix \left( (f^{-1}g)^n \right)
  \end{equation*}
  for all $n \in \N$.
  Therefore
  \begin{align*}
    S_{f,g}(z) \ &= \ \exp \left( \sum_{n=1}^\infty \frac{\# S(f^n,g^n)}{n} z^n \right)
                  \ = \ \exp \left( \sum_{n=1}^\infty \frac{\# \Fix\left((g^{-1}f)^n\right)}{n} z^n \right) \\
                 &= \ \zeta_{g^{-1}f}(z) \ = \ \zeta_{f^{-1}g}(z) .
  \end{align*}
  Now if $(g^{-1}f)$ or $(f^{-1}g)$ is pseudo-Anosov, then Theorem~\ref{thm:pseudo-Anosov-AM-rational}
  gives us rationality of $S_{f,g}(z)$.
\end{proof}

Theorem \ref{thm:Manning2} and theorem \ref{thm:pseudo-Anosov-AM-rational2} can be generalised to the following
\begin{thm}\label{thm:pseudo-Anosov-AM-rational3}
 Let $f,g: M \to M$ be two commuting homeomorphisms of  compact manifold without boundary
 such that $f^{-1}g$, or equivalently  $g^{-1}f$,  is tame and has rational Artin-Mazur zeta function
 $\ \zeta_{f^{-1}g}(z) $ then the synchronization zeta function  $S_{f,g}(z)$ is rational function and
 $S_{f,g}(z) =  \zeta_{f^{-1}g}(z)=  \zeta_{g^{-1}f}(z) $.
 \end{thm}
\begin{proof}
 The proof is indicated in the theorem \ref{thm:pseudo-Anosov-AM-rational2}.
\end{proof}

\subsection{Synchronization zeta function on finite sets}

\begin{thm} \label{thm:synchronization-commuting-permutations}
  Let X be a finite set and $\sigma_1, \sigma_2: X \to X$ two commuting bijections,
  i.e. commuting permutations.
  Then the synchronization zeta function $S_{\sigma_1, \sigma_2}(z)$ is rational.
\end{thm}
\begin{proof}
  For $\sigma_1^n(x) = \sigma_2^n(x)$ iff $(\sigma_2^{-n} \circ \sigma_1^n)(x) = x$, we have
  $\# S(\sigma_1^n, \sigma_2^n) = \# \Fix\left((\sigma_2^{-1}\sigma_1)^n\right)$.
  Let us denote the orbit of $x \in X$ under $\sigma_2^{-1}\sigma_1$ by $[x]$ and its order by $\#[x]$.
  Then
  \begin{equation*}
    \# \Fix\left((\sigma_2^{-1}\sigma_1)^n\right) = \sum_{\substack{[x] \\ \#[x] | n}} \#[x]
  \end{equation*}
  and
  \begin{align} \label{eq:finite-set-synchrozniation}
  \begin{split}
    S_{\sigma_1, \sigma_2}(z) \ &= \ \exp \left( \sum_{n=1}^\infty \frac{\# S(\sigma_1^n, \sigma_2^n)}{n} z^n \right)
      \ = \ \exp \left( \sum_{[x]} \sum_{\substack{n=1 \\ \#[x]|n}}^\infty \frac{\#[x]}{n} z^n \right) \\
    &= \ \prod_{[x]} \exp \left( \sum_{n=1}^\infty \frac{\#[x]}{\#[x]n} z^n \right)
      \ = \ \prod_{[x]} \exp \left( \sum_{n=1}^\infty \frac{1}{n} z^{\#[x]n} \right) \\
    &= \ \prod_{[x]} \exp \left( -\log \left( 1 - z^{\#[x]} \right)\right)
      \ = \ \prod_{[x]} \frac{1}{1 - z^{\#[x]}}.
  \end{split}
  \end{align}
  As $X$ is finite, the product is finite as well, and $S_{\sigma_1, \sigma_2}(z)$ is rational.
\end{proof}

\begin{rmk} \label{thm:example-synchronization-finite-algebraic}
  Without the commutativity assumption the theorem is not true in general.
  Let $X = \{0, 1, 2\}$ and
  \begin{equation*}
    \sigma_1 = \begin{pmatrix}0 & 1 & 2 \\ 0 & 2 & 1\end{pmatrix}, \qquad
    \sigma_2 = \begin{pmatrix}0 & 1 & 2 \\ 1 & 0 & 2\end{pmatrix}.
  \end{equation*}
  We have $\sigma_1 \left( \sigma_2(0) \right) = 2 \neq 1 = \sigma_2 \left( \sigma_1(0) \right)$.
  Observe that
  \begin{align*}
    \# S(\sigma_1, \sigma_2) = 0, \quad \# S(\sigma_1^2, \sigma_2^2) = \# S(\Id, \Id) = \# X = 3.
  \end{align*}
  Therefore the synchronization zeta function
  \begin{align*}
    S_{\sigma_1, \sigma_2}(z) \ &= \ \exp \left( \sum_{n=1}^\infty \frac{3}{2} \cdot \frac{(z^2)^n}{n} \right)
      \ = \ \exp \left( -\frac{3}{2} \log(1-z^2) \right) \\
    &= \ \frac{1}{\sqrt{(1-z^2)^3}}
  \end{align*}
  is \emph{algebraic}, not rational.
\end{rmk}

\begin{thm}
  Let $X$ be a finite set and $\sigma_1, \sigma_2: X \to X$ two commuting bijections, i.e. commuting permutations.
  The synchronization zeta function has the following functional equation
  \begin{equation*}
    S_{\sigma_1, \sigma_2} \left( \frac{1}{z} \right) \ = \ (-1)^p z^q S_{\sigma_1, \sigma_2}(z)
  \end{equation*}
  where $q$ is the number of periodic elements of $\sigma_2^{-1}\sigma_1$
  and $p$ is the number of periodic orbits of $\sigma_2^{-1}\sigma_1$.
\end{thm}
\begin{proof}
  Using \eqref{eq:finite-set-synchrozniation} we write
  \begin{align*}
    S_{\sigma_1, \sigma_2}\left(\frac{1}{z}\right) \ &= \ \prod_{[x]} \frac{1}{1 - z^{-\#[x]}}
      \ = \ \prod_{[x]} \frac{z^{\#[x]}}{z^{\#[x]} - 1} \\
    &= \ \prod_{[x]} \frac{-z^{\#[x]}}{1 - z^{\#[x]}}
      \ = \ \prod_{[x]} -z^{\#[x]} \cdot \prod_{[x]} \frac{1}{1 - z^{\#[x]}} \\
    &= \ \prod_{[x]} -z^{\#[x]} \cdot S_{\sigma_1, \sigma_2}(z).
  \end{align*}
  The statement follows because $\sum_{[x]} \# [x] = q$.
\end{proof}

\begin{thm} \label{thm:synch-numbers-permutations-finite-set}
  Let $\sigma_i, \sigma_j : X \to X$ be two bijections, i.e. permutations of a finite set.
  Then the following hold
  \begin{enumerate}
  \item[\rm (i)] The sequence of synchronization numbers $\{\#S(\sigma_i^n, \sigma_j^n)\}_{n=1}^\infty$ is periodic.
  \item[\rm (ii)] If $L$ is the length of the period of the sequence of synchronization numbers
      then for $k=1,\ldots,\lfloor L/2 \rfloor$
      \begin{equation*} 
        \#S(\sigma_i^k, \sigma_j^k) = \#S(\sigma_i^{L-k}, \sigma_j^{L-k})
      \end{equation*}
  \end{enumerate}
\end{thm}
\begin{proof}
(i) Let $k_i, k_j$ be ranks of $\sigma_i, \sigma_j$ in the group of permutations of $X$, i.e.
  \begin{equation*}
    \sigma_i^{k_i} = \Id = \sigma_j^{k_j}.
  \end{equation*}
  Denote $L := \LCM(k_i,k_j)$. This is the least natural number for which both $\sigma_i^L = \sigma_j^L = \Id$.
  It is clear that $\#S(\sigma_i^n,\sigma_j^n) = \#S(\sigma_i^{n+L},\sigma_j^{n+L})$ for $n \in \N$.

(ii) Observe that if $\sigma_i(x) = \sigma_j(x)$ then
  \begin{equation*}
    \sigma_j^{-1} \left( \sigma_i(x) \right) =  \sigma_j^{-1} \left( \sigma_j(x) \right) = x
    = \sigma_i^{-1} \left( \sigma_i(x) \right)
  \end{equation*}
  therefore synchronization points of $\sigma_i,\sigma_j$ are in bijection with synchronization points
  of $\sigma_i^{-1},\sigma_j^{-1}$ and $\#S(\sigma_i^n,\sigma_j^n) = \#S(\sigma_i^{-n},\sigma_j^{-n})$.
  The result follows from periodicity.
\end{proof}

\begin{thm} \label{thm:sychronization-zeta-log-derivative-rational}
  Let $\sigma_i,\sigma_j: X \to X$ be two bijections, i.e. permutations, of a finite set.
  Then the logarithmic derivative of the synchronization zeta function $S_{\sigma_i, \sigma_j}(z)$
  is rational and is equal
  \begin{align} \label{eq:synchronization-zeta-log-derivative}
    \left[ \ln S_{\sigma_i, \sigma_j}(z) \right]' = \frac{S_{\sigma_i, \sigma_j}'(z)}{S_{\sigma_i, \sigma_j}(z)}
      = \frac{\sum_{n=1}^{L} \#S(\sigma_i^n, \sigma_j^n) z^{n-1}}{1-z^L}.
  \end{align}
\end{thm}
\begin{proof}
  Direct computation gives us
  \begin{align*}
  \begin{split}
    \frac{S_{\sigma_i, \sigma_j}'(z)}{S_{\sigma_i, \sigma_j}(z)}
      &= \frac{\exp \left( \sum_{n=1}^\infty \frac{\#S(\sigma_i^n, \sigma_j^n)}{n} z^n \right)
      \cdot \sum_{n=1}^\infty \#S(\sigma_i^n, \sigma_j^n) z^{n-1}}
      {\exp \left( \sum_{n=1}^\infty \frac{\#S(\sigma_i^n, \sigma_j^n)}{n} z^n \right)} \\
    &= \sum_{n=1}^\infty \#S(\sigma_i^n, \sigma_j^n) z^{n-1}.
  \end{split}
  \end{align*}
  Since for $n = L+1, L+2, \ldots$ we have $\#S(\sigma_i^n, \sigma_j^n) - \#S(\sigma_i^{n-L}, \sigma_j^{n-L}) = 0$,
  another direct check gives us
  \begin{align*}
  \begin{split}
    (1 - z^L) \cdot \sum_{n=1}^\infty \#S(\sigma_i^n, \sigma_j^n) z^{n-1}
      &= \sum_{n=1}^L \#S(\sigma_i^n, \sigma_j^n) z^{n-1}
      + \sum_{n=L+1}^\infty \left( \#S(\sigma_i^n, \sigma_j^n)
      - \#S(\sigma_i^{n-L}, \sigma_j^{n-L}) \right) z^{n-1} \\
    &= \sum_{n=1}^L \#S(\sigma_i^n, \sigma_j^n) z^{n-1}
  \end{split}
  \end{align*}
  what, after dividing by $(1-z^L)$, gives us \eqref{eq:synchronization-zeta-log-derivative}.
\end{proof}

\begin{thm} \label{thm:synchronization-zeta-form-bijections-finite}
  Let $X$ be a finite set and $\sigma_1, \sigma_2: X \to X$ two bijections, i.e. permutations.
  Let $k_1, k_2 \in \N$ be such that $\sigma_1^{k_1} = \Id = \sigma_2^{k_2}$
  and denote $L := \LCM(k_1,k_2)$.
  Let $\omega := \exp(2\pi\iota/L)$ be the $L$-th root of $1$ and $\Re(\omega)$ its real part.

  Then the synchronization zeta function is of the form
  \begin{align} \label{eq:synchronization-permutations}
    S_{\sigma_1,\sigma_2}(z) = \begin{cases}
      (1-z)^{A_0} \cdot \prod_{k=1}^{\lfloor L/2 \rfloor} \left( z^2 - 2\Re(\omega^k)z+1 \right)^{A_k}, & 2 \!\not|\, L \\[1ex]
      (1-z)^{A_0}(1+z)^{A_{L/2}} \cdot \prod_{k=1}^{L/2-1} \left( z^2 - 2\Re(\omega^k)z+1 \right)^{A_k}, & 2 \,\vert\, L
    \end{cases},
  \end{align}
  where all $A_k \in \R$.
\end{thm}
\begin{proof}
  When we consider the logarithmic derivative of the synchronization zeta function formally, we can write
  \begin{align*}
    S_{\sigma_i, \sigma_j}(z) &= \exp \left( \int \frac{S_{\sigma_i, \sigma_j}'(z)}{S_{\sigma_i, \sigma_j}(z)} dz \right).
  \end{align*}
  We will integrate the logarithmic derivative using the form of Theorem~\ref{thm:sychronization-zeta-log-derivative-rational}.
  For simplicity, denote the numerator $\sum_{n=1}^{L} \#S(\sigma_i^n, \sigma_j^n) z^{n-1} =: P(z) \in \Z[z]$.
  The denominator has $L$ different roots which are $L$-th roots of $1$, i.e. $\omega^k, k=0,1,\ldots,L-1$
  therefore we factorize
  \begin{align*}
    z^L-1 = (z-1)(z-\omega)\cdots(z-\omega^{L-1}).
  \end{align*}
  Performing partial fraction decomposition we get
  \begin{align*}
    \frac{-P(z)}{z^L-1} &= \sum_{k=0}^{L-1} \frac{C_k}{z-\omega^k}
  \end{align*}
  and, after reducing the right side to a common denominator and denoting by $\widehat{(z-\omega^k)}$
  the missing term in the product,
  \begin{align*}
    \sum_{k=0}^{L-1} \frac{C_k}{z-\omega^k}
    &= \frac{\sum_{k=0}^{L-1} C_k \cdot (z-\omega^0) \cdots \widehat{(z-\omega^k)} \cdots (z-\omega^{L-1})}{z^L-1}.
  \end{align*}
  Observe that the derivative
  \begin{align*}
    (z^L-1)' = \left[ \prod_{k=0}^{L-1} (z-\omega^k) \right]'
      = \sum_{k=0}^{L-1} (z-\omega^0) \cdots \widehat{(z-\omega^k)} \cdots (z-\omega^{L-1})
  \end{align*}
  therefore, comparing numerators and substituting the roots one by one, for $k=0,\ldots,L-1$ we have
  \begin{align*}
    -P(\omega^k) &= C_k \cdot (z^L-1)'(\omega^k) = C_k \cdot L\omega^{k(L-1)} \\
    C_k &= \frac{-P(\omega^k)}{L\omega^{k(L-1)}}
  \end{align*}

  We are going to show that all $C_k$ are in fact real.
  If $L$ is even, then $\omega^{L/2} = \exp((2\pi\iota/L) \cdot (L/2)) = -1$
  is a single real root which we will consider separately.
  Regardless of parity of $L$, $1$ always is a single real root.
  We therefore have
  \begin{align*}
    C_0 = \frac{-P(1)}{L} \in \Q, \quad C_{L/2} = \frac{-P(-1)}{\pm L} \in \Q \text{ (for $L$ even)}.
  \end{align*}
  Now consider $L$ odd and $k$ such that $\omega^k \notin \R$. The imaginary part
  \begin{align*}
    \Im C_k &= \Im \left( -\frac{1}{L} \sum_{n=1}^L \#S(\sigma_i^n,\sigma_j^n) \omega^{k(n-1)+k(1-L)} \right) \\
    &= -\frac{1}{L} \sum_{n=1}^L \#S(\sigma_i^n,\sigma_j^n) \Im \left(\omega^{k}\right)^n \\
  \intertext{Note that $(\omega^k)^L=1$ and for $n=1,\ldots,\lfloor L/2 \rfloor$ the terms appears in complex
      conjugate pairs $\Im(\omega^{k})^n = -\Im(\omega^{k})^{L-n}$ so that we can rewrite the last sum}
    &= -\frac{1}{L} \sum_{n=1}^{\lfloor L/2 \rfloor} \left[\#S(\sigma_i^n,\sigma_j^n)
      \Im \left(\omega^{k}\right)^n - \#S(\sigma_i^{L-n},\sigma_j^{L-n}) \Im \left(\omega^{k}\right)^n \right] \\
  \intertext{Using Theorem~\ref{thm:synch-numbers-permutations-finite-set} (ii) we get}
    &= -\frac{1}{L} \sum_{n=1}^{\lfloor L/2 \rfloor} \#S(\sigma_i^n,\sigma_j^n)
      \left[ \Im\left(\omega^k\right)^n - \Im\left(\omega^k\right)^n \right] \\
    &= 0
  \end{align*}
  therefore $C_k \in \R$ is in fact real.

  Because $z^L-1$ is a polynomial with integer coefficients, all its complex roots appear in pairs
  with their complex conjugates $\overline{\omega^k} = \omega^{-k} = \omega^{L-k}$.
  Grouping conjugate complex roots we get
  \begin{align*}
    \frac{-P(z)}{z^L-1} &= \begin{cases}
      \frac{C_0}{z-1} + \sum_{k=1}^{\lfloor L/2 \rfloor} \left( \frac{C_k}{z-\omega^k} + \frac{\overline{C_k}}{z-\overline{\omega^k}} \right), & 2 \!\not|\, L \\[2ex]
      \frac{C_0}{z-1} + \frac{C_{L/2}}{z+1} + \sum_{k=1}^{L/2-1} \left( \frac{C_k}{z-\omega^k} + \frac{\overline{C_k}}{z-\overline{\omega^k}} \right), & 2 \,|\, L
    \end{cases} \\[2ex]
    &= \begin{cases}
      \frac{C_0}{z-1} + \sum_{k=1}^{\lfloor L/2 \rfloor} \frac{2C_kz-2\Re(C_k\overline{\omega^k})}{z^2-2\Re(\omega^k)z+1}, & 2 \!\not|\, L \\[2ex]
      \frac{C_0}{z-1} + \frac{C_{L/2}}{z+1} + \sum_{k=1}^{L/2-1} \frac{2C_kz-2\Re(C_k\overline{\omega^k})}{z^2-2\Re(\omega^k)z+1}, & 2 \,|\, L
    \end{cases}
  \end{align*}
  Let us take a closer look at the terms under the summation signs.
  Since $C_k \in \R$ we have $\Re(C_k\overline{\omega^k})/C_k = \Re(\omega^k)$
  so
  \begin{align*}
    \frac{2C_kz-2\Re(C_k\overline{\omega^k})}{z^2-2\Re(\omega^k)z+1}
      &= C_k \cdot \frac{2z-2\Re(\omega^k)}{z^2 -2\Re(\omega^k)z +1} \\
    C_k \int \frac{\left( z^2 -2\Re(\omega^k)z +1 \right)'}{z^2-2\Re(\omega^k)z+1} dz
      &= C_k \cdot \ln \left( z^2-2\Re(\omega^k)z+1 \right)
  \end{align*}
  Integrating term by term we get
  \begin{align*}
    \int \frac{S_{\sigma_i, \sigma_j}'(z)}{S_{\sigma_i, \sigma_j}(z)}
      &= \int \frac{-P(z)}{z^L-1} \\
    &= \begin{cases}
      C_0\ln(z-1) + \sum_{k=1}^{\lfloor L/2 \rfloor} C_k \ln \left( z^2-2\Re(\omega^k)z+1 \right)  , & 2 \!\not|\, L \\[2ex]
       C_0\ln(z-1) + C_{L/2}\ln(z+1) + \sum_{k=1}^{L/2-1} C_k \ln \left( z^2-2\Re(\omega^k)z+1 \right) , & 2 \,|\, L
    \end{cases}
  \end{align*}
  Taking exponent and setting $A_k := C_k$ provides us with the result.
\end{proof}

\begin{col} \label{thm:synchronization-zeta-finite-bijections-col}
  If all $A_k$ in \eqref{eq:synchronization-permutations} are integers then the synchronization zeta function
  is rational, which is the case for example for \emph{commuting} permutations
  as proved in {\rm Theorem~\ref{thm:synchronization-commuting-permutations}}.
  If all $A_k$ are rational then the synchronization zeta function is algebraic,
  which was illustrated in {\rm Remark~\ref{thm:example-synchronization-finite-algebraic}}.

  In the most general situation $A_k$ is real. Even if the synchronization zeta function is not rational or algebraic,
  the equation \eqref{eq:synchronization-permutations} uses polynomials and real exponents and this
  assures that the set of singularities of the synchronization zeta function does not have accumulation points,
  hence, the synchronization zeta function itself does not have the natural boundary.
\end{col}

More generally, for maps of a finite sets that not necessarily are bijections,
we have similar yet more complicated results.

\begin{lem} \label{thm:synchronization-finite-arbitrary-lemma}
  Let $X$ be a finite set and $\sigma_1, \sigma_2: X \to X$ two arbitrary maps. Then
  \begin{enumerate}
  \item[\rm (i)] The sequence of synchronization numbers of $\sigma_1, \sigma_2$ is periodic
    from a certain place, i.e. there exist $n, L \in \N$ such that
    \begin{align} \label{eq:synchronization-sequence-finite-periodic}
      \# S(\sigma_1^{n+k+L}, \sigma_2^{n+k+L}) = \# S(\sigma_1^{n+k}, \sigma_2^{n+k}), \quad \text{ for all } k =0,1,2,\ldots;
    \end{align}
  \item[\rm (ii)] We have the equality
    \begin{align*}
      \sum_{k=1}^\infty \# S(\sigma_1^k, \sigma_2^k) z^{k-1} = \frac{P(z)}{1-z^L},
    \end{align*}
    where $P(z) \in \Z[Z]$ is a polynomial of $\deg P = n+L-2$ with integer coefficients.
  \end{enumerate}
\end{lem}
\begin{proof}
  (i) Consider iterations $\sigma_1^k$ for $k=1,2,\ldots$ and let $x \in X$.
  For $X$ is finite, there is the first iteration $k_x$ such that $\sigma_1^{k_x}(x)$ is the same element
  as $\sigma_1^{k_x-l_x}(x)$ for some $l_x \in \N$. Denote $k_x-l_x=:n_x$.
  If we denote $n_1 := \max_x n_x, L_1 := \LCM\{l_x \;|\; x \in X\}$ then
  \begin{align*}
    \sigma_1^{n_1+k} = \sigma_1^{n_1+k+L_1}, \quad \text{ for all } k \in \N.
  \end{align*}
  In the same way we obtain $n_2, L_2$ for $\sigma_2$.
  We have \eqref{eq:synchronization-sequence-finite-periodic} for $n := \max(n_1, n_2), L:=\LCM(L_1, L_2)$.

  (ii) Direct computation gives us
  \begin{align*}
    (1-z^L) \cdot \sum_{k=1}^\infty \# S(\sigma_1^k, \sigma_2^k) z^{k-1}
      &= \sum_{k=1}^L \# S(\sigma_1^k, \sigma_2^k) z^{k-1}
      + \sum_{k=L+1}^\infty \left( \# S(\sigma_1^k, \sigma_2^k) - \# S(\sigma_1^{k-L}, \sigma_2^{k-L}) \right) z^{k-1}.
  \end{align*}
  Observe that the second sum is in fact finite. Indeed, if $n$ is as in (i)
  then we have \eqref{eq:synchronization-sequence-finite-periodic}
  and from $k=n+L$ all differences equal zero.
  If we denote the right side by $P(z)$ and divide both sides by $(1-z^L)$ then we get the result.
\end{proof}

\begin{thm}
  Let $X$ be a finite set and $\sigma_1, \sigma_2: X \to X$ two arbitrary maps.

  Denote by $L$ the length of the period of the sequence of synchronization numbers and denote
  \begin{align*}
    \omega:=\exp(2\pi \iota / L).
  \end{align*}
  Let $P(z) \in \Z[z]$ be the polynomial with integer coefficients such that
  \begin{align*}
    \left(1 - z^L \right) \cdot \sum_{n=1}^\infty \# S(\sigma_1^n, \sigma_2^n) z^{n-1} = P(z)
  \end{align*}
  as in {\rm Lemma~\ref{thm:synchronization-finite-arbitrary-lemma}(ii)}.
  Let $Q(z), R(z) \in \Z[z]$ be the polynomials fulfilling
  \begin{align*}
    P(z) = Q(z) \cdot (1-z^L) + R(z), \quad \deg R < L.
  \end{align*}

  Then the synchronization zeta function is of the form
  \begin{align*}
    S_{\sigma_1,\sigma_2}(z) = \exp \left( \int Q(z) dz \right) \cdot (1-z)^{A_0} \cdot T(z)
  \end{align*}
  where
  \begin{align*}
    T(z) = \begin{cases}
      \prod_{k=1}^{\lfloor L/2 \rfloor} \left( z^2 - 2\Re(\omega^k)z+1 \right)^{A_k}
      \cdot \left( \frac{-\Re(\omega^k) + \iota\sqrt{1-\Re(\omega^k)^2} \ + \ z}{\Re(\omega^k) + \iota\sqrt{1-\Re(\omega^k)^2} \ - \ z} \right)^{B_k \iota}, & 2 \!\not|\, L \\[1ex]
      (1+z)^{A_{L/2}} \cdot
      \prod_{k=1}^{L/2-1} \left( z^2 - 2\Re(\omega^k)z+1 \right)^{A_k}
      \cdot \left( \frac{-\Re(\omega^k) + \iota\sqrt{1-\Re(\omega^k)^2} \ +\ z}{\Re(\omega^k) + \iota\sqrt{1-\Re(\omega^k)^2} \ -\ z} \right)^{B_k \iota}, & 2 \,|\, L
     \end{cases}
  \end{align*}
  where $A_i, B_i \in \R$ and $\iota^2 = -1$.
\end{thm}
\begin{proof}
  We are going to follow the same path as in Theorem~\ref{thm:synchronization-zeta-form-bijections-finite}.
  Two important differences are that we do not have the property of Theorem~\ref{thm:synch-numbers-permutations-finite-set} (ii)
  and $\deg P$ can be larger than the length $L$ of the period in the sequence of synchronization numbers.

  Integrating term by term we get
  \begin{align*}
  \begin{split}
    S_{\sigma_1, \sigma_2}(z) &= \exp \left( \sum_{n=1}^\infty \frac{\# S(\sigma_1^n, \sigma_2^n)}{n} z^n \right) \\
    &= \exp \left( \int \sum_{n=1}^\infty \# S(\sigma_1^n, \sigma_2^n) z^{n-1} dz \right) \\
    &= \exp \left( \int \frac{P(z)}{1-z^L} dz \right) \\
    &= \exp \left( \int Q(z) dz \right) \cdot \exp \left( \int \frac{R(z)}{1-z^L} dz \right).
  \end{split}
  \end{align*}
  Because $\deg R < L$ we can perform partial fraction decomposition as in the proof
  of Theorem~\ref{thm:synchronization-zeta-form-bijections-finite}:
  \begin{align*}
    \frac{-R(z)}{z^L-1} = \sum_{k=0}^{L-1} \frac{C_k}{z-\omega^k}, \quad \text{ where } C_k = \frac{-R(\omega^k)}{L\omega^{k(L-1)}} \in \C.
  \end{align*}
  In the sum, roots of unity appear in conjugate pairs, excluding $-1$ for $L$ even and $1$ regardless of parity of $L$.
  Grouping the complex conjugate roots we get (for $L$ odd)
  \begin{align*}
  \begin{split}
    \sum_{k=1}^{L-1} \frac{C_k}{z-\omega^k} &= \sum_{k=1}^{\lfloor L/2 \rfloor}
      \frac{C_k(z-\overline{\omega^k})+\overline{C_k}(z-\omega^k)}{z^2 - 2\Re(\omega^k)z +1} \\
    &= \sum_{k=1}^{\lfloor L/2 \rfloor} \frac{2\Re(C_k)z - 2\Re(C_k\overline{\omega^k})}{z^2 - 2\Re(\omega^k)z +1} \\
    &= \sum_{k=1}^{\lfloor L/2 \rfloor} \Re(C_k) \frac{2z - 2\Re(C_k\overline{\omega^k})/\Re(C_k)}{z^2 - 2\Re(\omega^k)z +1}.
  \end{split}
  \end{align*}
  For fixed $k$ we can write
  \begin{align*}
  \begin{split}
    \frac{2z - 2\Re(C_k\overline{\omega^k})/\Re(C_k)}{z^2 - 2\Re(\omega^k)z +1} &=
      \frac{2z -2\Re(\omega^k) +2\Re(\omega^k) -2\Re(C_k\overline{\omega^k})/\Re(C_k)}{z^2 - 2\Re(\omega^k)z +1} \\
    &= \frac{2z -2\Re(\omega^k)}{z^2 - 2\Re(\omega^k)z +1}
      + \left( 2\Re(\omega^k) -2\Re(C_k\overline{\omega^k})/\Re(C_k) \right)\frac{1}{z^2 - 2\Re(\omega^k)z +1}.
  \end{split}
  \end{align*}
  In the first term the numerator is the derivative of the denumerator, therefore after integration
  it will give the logarithm
  \begin{align*}
    \int \frac{2z -2\Re(\omega^k)}{z^2 - 2\Re(\omega^k)z +1} dz = \ln \left( z^2 - 2\Re(\omega^k)z +1 \right)
  \end{align*}
  which will be canceled out after exponentiation.
  The second term, if $\Im C_k \neq 0$, after integration gives rise to
  \begin{align*}
    \int \frac{1}{z^2-2\Re(\omega^k)z+1} dz =
      \frac{1}{\sqrt{1-\Re(\omega^k)^2}} \arctan \left( \frac{z-\Re(\omega^k)}{\sqrt{1-\Re(\omega^k)^2}} \right)
  \end{align*}
  which, after using the known identity
  \begin{align*}
    \arctan(z) = \frac{\iota}{2} \ln \left(\frac{\iota + z}{\iota - z} \right),
  \end{align*}
  gives us
  \begin{align*}
    \int \frac{1}{z^2-2\Re(\omega^k)z+1} dz =
      \frac{\iota}{2\sqrt{1-\Re(\omega^k)^2}} \ln \left( \frac{-\Re(\omega^k) + \iota\sqrt{1-\Re(\omega^k)^2} + z}{\Re(\omega^k) + \iota\sqrt{1-\Re(\omega^k)^2} - z} \right).
  \end{align*}

  Hence, for a fixed $k$, we have
  \begin{align*}
  \begin{split}
    \exp \left( \int \frac{C_k}{z-\omega^k} dz \right)
      = \exp &\left[ \Re(C_k) \ln \left( z^2 - 2\Re(\omega^k)z +1 \right) \right] \\
    &\cdot \exp \left[ \frac{\Re(C_k)\Re(\omega^k) - \Re(C_k\overline{\omega^k})}{\sqrt{1-\Re(\omega^k)^2}} \cdot \iota \cdot \ln \left( \frac{-\Re(\omega^k) + \iota\sqrt{1-\Re(\omega^k)^2} + z}{\Re(\omega^k) + \iota\sqrt{1-\Re(\omega^k)^2} - z} \right) \right].
  \end{split}
  \end{align*}
  Summing over $k$ and setting
  \begin{align*}
    A_k := \Re(C_k), \quad B_k := \frac{\Re(C_k)\Re(\omega^k) - \Re(C_k\overline{\omega^k})}{\sqrt{1-\Re(\omega^k)^2}}
  \end{align*}
  finishes the proof.
\end{proof}

\begin{col}
  In analogy to the {\rm Corollary~\ref{thm:synchronization-zeta-finite-bijections-col}} we observe
  that the synchronization zeta function of two arbitrary maps of a finite set does not have a natural boundary.
  That is because all terms can be written as real powers of polynomials and complex powers of rational functions,
  hence, the set of singularities does not have accumulation points.

  Observe that if all $B_k = 0$, i.e. all $\C_k \in \R$, then this general form collapses
  to the form of {\rm Theorem~\ref{thm:synchronization-zeta-form-bijections-finite}}
  with an additional exponent of an integral of a polynomial term.
\end{col}

\vspace{1cm}
\section{Gauss congruences for numbers of synchronization points} \label{sec:congruences}

Let $\mu(n)$, $n\in\N$, be the {\em M\"obius function}, i.e.
\begin{align*}
  \mu(n) = \left\{
  \begin{array}{ll}
    1, & \mbox{ if } n=1,  \\
    (-1)^k, & \mbox{ if }  n  \mbox{ is a product of } k \mbox{ distinct primes,}\\
    0, & \mbox{ if }  n  \mbox{ is not square-free.}
  \end{array}
  \right.
\end{align*}
In number theory, the following Gauss congruence for integers holds:
\begin{align*}
  \sum_{d\mid n}\mu(n/d)\cdot a^{d}\equiv 0\mod n
\end{align*}
for any integer $a$ and any natural number $n$. In the case of a prime power $n=p^r$, the Gauss congruence turns into the Euler congruence. Indeed, for $n=p^r$ the M\"obius function $\mu(n/d)=\mu(p^r/d)$ is nonzero only in two cases: when $d=p^r$ and when $d=p^{r-1}$. Therefore, from the Gauss congruence we obtain the Euler congruence
\begin{align*}
  a^{p^r}\equiv a^{p^{r-1}}\mod p^r.
\end{align*}
When $(a,n)=1, n=p^r$, these congruences are equivalent to the following classical Euler's theorem:
\begin{align*}
  a^{\varphi(n)}\equiv 1\mod n.
\end{align*}
These congruences have been generalized from integers to some other mathematical invariants such as the traces of powers of all integer matrices $A$ and the Lefschetz numbers of iterations  of a map {see \cite{mp99,Z}}:
\begin{align} \label{Gauss}
&\sum_{d\mid n}\mu(n/d)\cdot \tr(A^{d})\equiv 0\mod n,\\
\label{Euler}
&\tr(A^{p^r})\equiv\tr(A^{p^{r-1}})\mod p^r.
\end{align}
A. Dold in \cite{Dold83} (see also \cite[Theorem 3.1.4]{JezMar}) proved by a geometric argument
the following congruences \eqref{Dold} for the Lefschetz numbers of iterations of a map $f$
on a compact ANR $X$ and any natural number $n$

\begin{align}
\label{Dold}
\sum_{d\mid n}\mu(n/d)\cdot L(f^{d})\equiv 0\mod n .\tag{DL}
\end{align}
These congruences are now called the Dold congruences.
It is also shown in \cite{mp99} (see also \cite[Theorem~9]{Z}) that the above congruences \eqref{Gauss},
\eqref{Euler} and \eqref{Dold} are equivalent. For more detailed discussion see \cite{ByszGraffWard}.

To prove Gauss congruences in case of rationality of the synchronization zeta function
we follow \cite[Theorem~2.1]{BaBo}, \cite[Theorem~3.1.23, Proposition~3.1.12]{JezMar},
and \cite{FelsSlo24}.

\begin{lem}[cf. \protect{\cite[Theorem~3.1.23]{JezMar}}, see also \protect{\cite[Lemma~6.2]{FelsSlo24}}]\label{thm:rational-zeta-coefficients-sum}
  $S_{\varphi, \psi}(z)$ is rational iff for every $k \in \N$ we have
  \begin{equation}\label{eq:coincidence-sum}
    \# S(\varphi^k, \psi^k) = \sum_{i=1}^{r} \chi_i \lambda_i^k
  \end{equation}
  where $r \in \N, \ \chi_i \in \Z$ and $\lambda_i \in \C$ are distinct algebraic integers.
\end{lem}

\begin{lem}[cf. \protect{\cite[Theorem~3.1.23]{JezMar}}, see also \protect{\cite[Lemma~6.3]{FelsSlo24}}]\label{thm:coefficients-sum-map-bouquet}
  There exists a map $f: X \to X$ of a compact euclidean neighborhood retract such that
  $L(f^k) = \# S(\varphi^k, \psi^k)$ for every $k \in \N$ iff
  the equality \eqref{eq:coincidence-sum} holds for every $k \in \N$
\end{lem}

The next lemma is a result of \cite[Theorem 2.1]{BaBo}.
We present a formulation more convenient for our purposes.

\begin{lem}[\protect{\cite[Proposition 3.1.12]{JezMar}}]\label{thm:integer-matrix-map-bouquet}
  Let X be the bouquet of $n_1$ circles and $n_2$ $2$-spheres, $n_1 \geq 2$.
  Then for each pair of matrices $A_e \in M_{n_2}(\Z), A_o \in M_{n_1-1}(\Z)$
  there exists a self-map $f: X \to X$ satisfying
  \[
    L(f^k) = \tr A_e^k - \tr A_o^k
  \]
  for every $k \in \N$.
\end{lem}

\begin{thm}\label{thm:rational-congruences}
  Let $\varphi, \psi$ be a synchronously tame pair of
  maps of a topological space.
  If the synchronization zeta function $S_{\varphi, \psi}(z)$ is rational,
  then for every $n \in \N$ we have the Gauss congruences
    \begin{equation*} 
      \sum_{k|n} \mu(n/k) \cdot \# S(\varphi^k, \psi^k) \equiv 0 \pmod n.
    \end{equation*}
\end{thm}
\begin{proof}
  Lemma~\ref{thm:rational-zeta-coefficients-sum} and Lemma~\ref{thm:coefficients-sum-map-bouquet}
  show that $S_{\varphi, \psi}(z)$ is rational iff
  there exists a map $f: X \to X$ of a compact euclidean neighborhood retract $X$ to itself such that
  for every $k \in \N$ we have
  \begin{equation} \label{eq:synchornization=lefschetz}
    L(f^k) = \# S(\varphi^k, \psi^k).
  \end{equation}
  Direct substitution to the Dold congruences \eqref{Dold} finishes the proof.
\end{proof}

\vspace{1cm}

\section{The growth rate of number of synchronization points and topological entropy} \label{sec:entropy}

Let $\alpha, \beta: X \to X$ be homeomorphisms of a compact metric space to itself.
In this section we establish connection between the growth rate $S^\infty(\alpha, \beta)$
and topological entropy for a class of pairs of homeomorphisms.

A homeomorphism $f$ of a compact metric space $(X,d)$ is \emph{expansive} if there exists a constant $\epsilon > 0$
such that given any two distinct points $x, y \in X$, there exists $ n \in \Z$ such that $d(f^n(x), f^n(y)) >\epsilon$.
We call such an $\epsilon$ an~\emph{expansive constant}.

We will need also a notion of the Bowen's \emph{specification} \cite{Bow}.
The formal definition can seem vague at first so it is good to have the following intuition in mind.
We say that a self-map $f: X \to X$ of a compact metric space satisfies \emph{specification}
if, given a precision $\epsilon>0$ and a~number of segments of orbits, one is able to find a~periodic orbit
which, with the precision $\epsilon$, follows each one of them and is moving from one segment to another
in a fixed amount of time which depends only on $\epsilon$.

Formally, we say that a homeomorphism $f$ satisfies \emph{specification} if for each $\epsilon > 0$
there is an integer $p(\epsilon)$ for which the following is true.
Given pairs $(x_1, I_1), \dots, (x_n, I_n)$, where $x_i \in X$ and $I_i$ are intervals of integers
contained in $[a,b]$ such that $d(I_i, I_j) \geq p(\epsilon)$ for $i \neq j$,
there exists an $x \in X$ for which $f^{b-a+p(\epsilon)}(x) = x$
and $d(f^kx, f^kx_i) < \delta$ for all $k \in I_i$ for all $i=1,\dots,n$.

Our result is based on \cite[Lemma~4]{Bow} which is far more general than we need here.
We follow logic of the proof but we tailor it to our purpose.
By $S(f,n,\epsilon)$ we denote the order of the maximal $(n,\epsilon)$-separeted set.
The topological entropy of a homeomorphism $f$ is denoted by $h(f)$.

\begin{lem}[\protect{cf. \cite[Lemma 4]{Bow}}] \label{thm:bowen-lemma4}
  Let $f: X \to X$ be an expansive homeomoprhism of a compact metric space satisfying specification.
  There are positive constants $A, B$ such that
  \begin{equation} \label{eq:bowen-lemma4}
    A \cdot \exp \big( h(f) \cdot n \big)
      \ \leq \ \#\Fix(f^n)
      \ \leq \ B \cdot \exp \big( h(f) \cdot n \big)
  \end{equation}
  for sufficiently large $n \in \N$.
\end{lem}

\begin{thm} \label{thm:growth=entropy}
  Let $\alpha, \beta: X \to X$ be a synchronously tame pair of commuting homeomorphisms of a compact metric space $X$
  such that $\beta^{-1}\alpha$ is expansive and satisfies specification.
  Then
  \begin{equation*}
    S^\infty(\alpha, \beta) \ = \ \exp \big( h(\beta^{-1}\alpha) \big).
  \end{equation*}
\end{thm}
\begin{proof}
  Directly from the definition of $S(\alpha, \beta)$ and commutativity
  \begin{align*}
    x \in S(\alpha^n, \beta^n) \ &\Leftrightarrow \ \alpha^n(x) = \beta^n(x)
    \ \Leftrightarrow \ (\beta^{-1}\alpha)^n(x) = x
    \ \Leftrightarrow \ x \in \Fix \left( (\beta^{-1}\alpha)^n \right).
  \end{align*}
  We substitute $\# S(\alpha^n, \beta^n)$ into \eqref{eq:bowen-lemma4} and take $n$-th root
  \begin{equation*}
    \sqrt[n]{A} \cdot \exp \big( h(\beta^{-1}\alpha) \big)
      \ \leq \ \sqrt[n]{\# S(\alpha^n, \beta^n)}
      \ \leq \ \sqrt[n]{B} \cdot \exp \big( h(\beta^{-1}\alpha) \big).
  \end{equation*}
  Taking the limit $n \to \infty$ provides us with the result.
\end{proof}

\begin{col}
  Let $\alpha, \beta$ be as in {\rm Theorem \ref{thm:growth=entropy}}
  and consider the synchronization zeta function $S_{\alpha,\beta}(z)$.
  The radius of its convergence $R$ is connected with topological entropy
  in the following way
  \begin{equation*}
    R \ = \ \frac{1}{\limsup_{n \to \infty} \sqrt{\frac{S(\alpha^n, \beta^n)}{n}}}
      \ = \ \frac{1}{S^\infty(\alpha, \beta)}
      \ = \ \exp \big(-h(\beta^{-1}\alpha) \big).
  \end{equation*}
\end{col}

\begin{rmk}
  The result of {\rm Theorem \ref{thm:growth=entropy}} is an analog of one of Miles \cite[Theorem~2.2]{Miles13}
  but we want to highlight the importance of the assumption of specification.
\end{rmk}

\vspace{1cm}

\section{Asymptotic behavior of the sequence $\{\# S(\varphi^k, \psi^k)\}$}  \label{sec:asymptotic}

Assume that Synchronization zeta function is rational.
Then the equality \eqref{eq:coincidence-sum}  $\# S(f^k, g^k) = \sum_{i=1}^{r} \chi_i \lambda_i^k$ holds
and from the proof of Lemma~\ref{thm:rational-zeta-coefficients-sum} we know
that $\lambda_i^{-1}$ are zeros and poles of the zeta function.
Define
\begin{equation*}
  \lambda(f,g) := \max_i \{|\lambda_i|\}, \qquad n(f,g) := \#\{ i \mid |\lambda_i| = \lambda(f,g) \}.
\end{equation*}
In Theorem~\ref{thm:synchronization-dichotomy} we proved a formula for growth
of synchronization points.
Providing $S_{f,g}(z)$ is rational we can say more about the structure of the set
of limit points of the sequence $\{S(f^k, g^k) / \lambda(f,g)^k\}_{k=1}^\infty$.
Inspired by the result of Babenko and Bogatyi \cite[Theorem 2.6]{BaBo} for Lefschetz numbers (see also \cite[Theorem~3.1.53]{JezMar})
and by the similar results \cite[Theorem 4.1]{FelsLee} for Nielsen numbers,
and \cite[Theorem~8.4]{FelsSlo24} for Reidemeister coincidence numbers,
we state the following

\begin{thm}\label{thm:synchronization-asymptotic-sequence}
  Let $f,g: X \to X$ be a synchronously tame pair of maps of a compact metric space $X$
  and let the synchronization zeta function $S_{f, g}(z)$ be rational.

  One of the three possibilities holds:
  \begin{enumerate}
    \item[\rm (i)] $\#S(f^k,g^k) = 0, k=1, 2, \dots$.
    \item[\rm (ii)] The sequence $\{ \# S(f^k, g^k) \ / \ \lambda(f, g)^k \}_{k=1}^\infty$ has the same
      limit points set as a periodic sequence $\{ \sum_{j=1}^{n(\varphi, \psi)} \alpha_j \epsilon_j^k \}_{k=1}^\infty$
      where $\alpha_j \in \Z, \ \epsilon_j \in \C$ and $\epsilon_j^q = 1$ for some integer $q>0$.
    \item[\rm (iii)] The set of limit points of the sequence $\{ \# S(f^k, g^k) \ / \ \lambda(f, g)^k \}_{k=1}^\infty$
      contains an interval.
  \end{enumerate}
\end{thm}
\begin{proof}
  The equality \eqref{eq:coincidence-sum}  $\# S(f^k, g^k) = \sum_{i=1}^{r} \chi_i \lambda_i^k$
  is a result of Lemma~\ref{thm:rational-zeta-coefficients-sum}.
  Observe that if $\lambda(f,g) = 0$ then obviously $\#S(f^k, g^k) = 0$ for all $k \in \N$.
  On the other hand, if $\#S(f^k, g^k) = 0$ for all $k \in \N$, then the synchrozniation zeta function
  $S_{f,g}(z) \equiv 1$. From the proof of Lemma~\ref{thm:rational-zeta-coefficients-sum} we know
  that every $\lambda_i$ is either a zero or a pole of the zeta function.
  The constant function has neither zeros nor poles, therefore $\lambda(f,g) = 0$.
  This is the case (i).

  Now assume that $\lambda(f,g) \neq 0$.
  We factor every $\lambda_i$ for which $|\lambda_i| = \lambda(\varphi, \psi)$
  in \eqref{eq:coincidence-sum} as
  $$\lambda_i = \lambda(\varphi, \psi) \exp(2\pi\iota\theta_i),$$
  where $\iota^2 = -1$ is the imaginary unit.
  Then
  \[
    \# S(f^k, g^k) = \lambda(f, g)^k
      \left( \sum_{|\lambda_i| = \lambda(f, g)} \chi_i \exp(2\pi\iota k \theta_i) \right)
      + \sum_{|\lambda_i| < \lambda(f, g)} \chi_i \lambda_i^k
  \]
  and
  \[
    \left| \frac{\# S(f^k, g^k)}{\lambda(f, g)^k}
      - \sum_{|\lambda_i| = \lambda(f, g)} \chi_i \exp(2\pi\iota k \theta_i) \right|
    \leq \sum_{|\lambda_i| < \lambda(f, g)} |\chi_i|
      \cdot \left| \frac{\lambda_i}{\lambda(f, g)} \right|^k.
  \]
  Right hand side goes to $0$ when $k \to \infty$ therefore the sequences
  \[
    \left\{ \# S(f^k, g^k) \ / \ \lambda(f, g)^k \right\}_{k=1}^\infty
    \mbox{\quad and \quad }
    \left\{ \sum_{|\lambda_i| = \lambda(f, g)} \chi_i \exp(2\pi\iota k \theta_i) \right\}_{k=1}^\infty
  \]
  have the same asymptotic behavior.
  Now we claim that the structure of the limit points set of the latter sequence
  depends on rationality of all $\theta_i$.

  Assume that for all $i = 1, \dots, n(f, g)$ we can write $\theta_i = p_i / q_i$ where $p_i, q_i \in \Z$.
  Put $q := \operatorname{lcm}(q_1, \dots, q_{n(f, g)})$.
  Whenever $k$ is a multiplicity of $q$, all $\exp(2\pi\iota k \theta_i) = 1$.
  This is the case (ii).

  Now assume that some of $\theta_i$ are irrational and denote the subset of their indices
  by $\mathcal{S} := \left\{ i_1, \dots, i_s \right\} \subset \{1, \dots, n(f, g)\}$.
  We can express the assumption as linear independence
  \[
    \left( \sum_{i \in \mathcal{S}} \alpha_i \theta_i = \alpha, \quad \alpha_i, \alpha \in \Z \right) \qquad
    \Longrightarrow \qquad \forall_i \ \alpha_i = \alpha = 0.
  \]
  Then the Kronecker theorem (see \cite[Theorem VIII.6]{Cha}) states that the sequence
  $\left\{ (k\theta_{i_1}, \dots, k\theta_{i_s}) \right\}_{k=1}^\infty$ is dense
  in the $s$-dimensional torus $\mathbb{T}^s$.

  Consider the continuous function
  \[
    h: \mathbb{T}^{n(f, g)} \to [0, \infty), \qquad
    h(\xi_1, \dots, \xi_{n(f, g)}) = \left| \sum_{i=1}^{n(f, g)} \chi_i \exp(2\pi\iota \xi_i) \right|.
  \]
  We can consider the subtorus
  \[
    \mathbb{T}_{\mathcal{S}} := \left\{ (\xi_1, \dots, \xi_{n(f, g)})
      \in \mathbb{T}^{n(f, g)} \mid \xi_j = 0, \forall j \notin \mathcal{S} \right\}
  \]
  For $\mathbb{T}_{\mathcal{S}}$ is compact and the restriction
  $h_\mathcal{S} := \left. h \right|_{\mathbb{T}_\mathcal{S}}$
  is continuous, there exists the maximum $m_\mathcal{S} := \max_{\mathbb{T}_{\mathcal{S}}} h_\mathcal{S}(\xi)$.
  Observe that the closure of the image of the dense sequence contains the interval
  \[
    \overline{h_\mathcal{S}\big( \left\{ k\theta_{i_1}, \dots, k\theta_{i_s} \right\}_k \big)}
    \ \supset \ [0, m_\mathcal{S}].
  \]
  This is the case (iii).
\end{proof}

\vspace{1cm}
\section*{Declaration of interest}
Declarations of interest: none.

\vspace{1cm}
\section*{Funding}
This research did not receive any specific grant from funding agencies in the public, commercial,
or not-for-profit sectors.


\end{document}